\documentclass[leqno]{siamonline190516}
\usepackage[T1]{fontenc}
\usepackage[utf8]{inputenc}
\usepackage{amsmath,amsfonts,amssymb,bbold,bbm}
\usepackage{pgf,tikz}
\usetikzlibrary{arrows}
\usetikzlibrary{positioning}
\usepackage{color}

\usepackage{soul}  
\numberwithin{theorem}{section}

\def\uno{\mathbb 1}
 
\def\b{\boldsymbol}

\newcommand{\Tau}{\mathcal{T}}
\renewcommand{\SS}{\mathcal{S}_1}
\newcommand{\SZ}{\mathcal{Z}_1}
\newcommand{\TP}{\boldsymbol{P}}
\newcommand{\R}{\mathbb{R}}
\newcommand{\II}{\mathcal{I}}

\newsiamremark{remark}{Remark}

\title{Ergodicity coefficients for higher-order stochastic processes\thanks{\vspace{-1em} \funding{The work of D.F.\ was supported by INdAM-GNCS, Italy, and by the departmental research project ICON (Innovative Combinatorial Optimization in Networks), DMIF-PRID 2017, University of Udine, Italy. The work of F.T.\ was funded by the European Union's Horizon 2020 research and innovation
programme under the Marie Sk\l odowska-Curie individual fellowship ``MAGNET'' grant agreement no.\ 744014}}}
\author{Dario Fasino\thanks{Department of Mathematics, Computer Science and Physics, University of Udine, Udine, Italy. (\email{dario.fasino@uniud.it})} \and Francesco Tudisco\thanks{School of Mathematics, Gran Sasso Science Institute (GSSI), 67100, L'Aquila, Italy (\email{francesco.tudisco@gssi.it})}}

\begin{document}
	
	\maketitle
	
\begin{abstract}
The use of higher-order stochastic processes such as nonlinear Markov chains or vertex-reinforced  random walks is significantly growing in recent years as they are much better at modeling 
high dimensional 
data and nonlinear dynamics 
in numerous application settings. 
In many cases of practical interest, these processes 
are identified with
a stochastic tensor and their stationary distribution is a tensor $Z$-eigenvector. However, fundamental questions such as the convergence of the process towards a limiting distribution and the uniqueness of such a limit are still not well understood and are the subject of rich recent literature. 
Ergodicity coefficients for stochastic matrices provide a valuable and widely used tool to analyze the long-term behavior of standard, first-order, Markov processes.   
%
%
%
In this work, we extend an important class of ergodicity coefficients 
to the setting of stochastic tensors. We show that the proposed higher-order ergodicity coefficients provide new explicit formulas that (a)  guarantee the   uniqueness of Perron $Z$-eigenvectors of stochastic tensors, (b) provide bounds on the sensitivity of such eigenvectors with respect to changes in the tensor and (c) ensure the  convergence of different types of  higher-order stochastic processes governed by cubical stochastic tensors. Moreover, we illustrate the advantages of the proposed ergodicity coefficients on several example application settings, including the analysis of PageRank vectors for triangle-based random walks and the convergence of lazy higher-order random walks.
\end{abstract}

\begin{keywords} 
Nonnegative tensors, Stochastic tensors, Higher-order Markov chain, Ergodicity coefficient, $Z$-eigenvector, Vertex reinforced random walk, Spacey random walk, Multilinear PageRank
\end{keywords}

\begin{AMS}
15B51, 
65F35, 
60J10, 
65C40  
\end{AMS}

\section{Introduction}
\label{sec:introduction}
Markov processes are among the best known and most popular stochastic processes in computational mathematics and mathematics of data science. 
For these types of processes the state transitions only depend  on the last state. This is modeled by a stochastic matrix $P$ whose entries $P_{ij}$ quantify the probability of the process of transitioning from state $j$ to state $i$. 
Any such a matrix leaves the simplex $\SS = \{x\geq 0: x^T \uno = 1\}$ invariant and the  classical Brouwer's fixed point theorem  thus implies that there exists at least one stationary distribution $x=Px$ for the Markov chain described by~$P$. 
In other words, there exists at least one eigenvector $x$ of $P$ corresponding to the eigenvalue $1$, such that  $x$ has nonnegative entries that sum up to one. 
While Brouwer's theorem holds in general for mappings leaving a closed convex set invariant, much more can be said for the specific case of stochastic matrices. In particular, if the Markov chain described by $P$ is ergodic, 
then $P$ has a unique nonnegative eigenvector $x$ in $\mathcal S_1$  which corresponds to the eigenvalue $1$, the magnitude of any other eigenvalue of $P$ is strictly smaller than one and the power method $x_{t+1} = Px_t$ converges to $x$, for any choice of $x_0\in \mathcal S_1$, with a convergence rate that depends on the largest sub-dominant eigenvalue. 
In a way, these properties characterize the concept of ergodic chain and the so-called ergodicity coefficients were introduced to estimate whether or not a Markov chain is ergodic without resorting to spectral properties \cite{Ergodic,SenetaBook}.

A natural extension of a Markov process is to have the state transitions depend on 
several previous states,
rather than just the last one. 
While the study and application of 
this kind of higher-order stochastic processes has a relatively long history, see e.g., \cite{Ben97,Raftery85}, their interest 
has grown significantly in more recent years due to their ability to improve the mathematical modeling  and understanding of numerous problems in data and network sciences, such as  detecting communities and analyzing spreading dynamics in networks \cite{rosvall2014memory,williams2019effects}, understanding the behavior of web browsers and drivers trajectories   \cite{BGL17,chierichetti2012web}, defining new clustering algorithms in data mining that exploit motifs and non-backtracking walks \cite{Benson2015TensorSC,krzakala2013spectral}, improving centrality and ranking models for networks and hypergraphs \cite{arrigo2019non,MLPR, mei2010divrank,Ng2011MultiRank} and forecasting the appearance of new links or finding missing links in networks \cite{AHTproceedings,nassar2019pairwise}.

Many higher-order stochastic processes of practical interest can be modeled by hypermatrices, or tensors, with $m$ modes $\b P = (\b P_{i_1,\dots,i_m})$. When $m=3$, for example, we have a second-order Markov chain if  $\b P_{ijk}$ quantifies the probability of transitioning to state $i$, given that the last state was $j$ and the previous one $k$. Another example is given by  linear vertex-reinforced random walks, where the transition probability from state $j$ to state $i$ is defined as $\sum_{k}\b P_{ijk}y_k$, for a vector $y$ which depends on the history of the states that have been  visited.
The stationary distribution equation for
these higher-order stochastic processes boils down to the $Z$-eigenvector of the corresponding stochastic tensor. 
While the existence of such stationary distribution is ensured also in this setting by Brouwer's fixed point theorem, the ergodicity of
those processes is much less understood than that of matrix-based processes. 
By extending the wide and influential literature on ergodicity coefficients for matrices, in this work we introduce a family of higher-order ergodicity coefficients for stochastic cubical  tensors and discuss how these allow us to derive new conditions on the existence, uniqueness and computability of stationary distributions for different type of higher-order stochastic processes described by tensors. In particular,  second-order Markov chains  and a new  class of linear vertex-reinforced random walks for which, to the best of our knowledge, we provide the first convergence result for both the occupation vector and the density distribution. 
This class includes previously considered vertex-reinforced stochastic processes such as the spacey random walk \cite{BGL17}.

From the linear algebraic perspective, our new conditions allow us to prove guarantees for existence, uniqueness and computability of the Perron $Z$-eigenvector of a stochastic tensor of order three. Dominant $Z$-eigenvectors of nonnegative tensors appear in a large variety of applications, including diffusion kurtosis imaging in medical engineering \cite{qi2008d}, low-rank factorization and signal processing \cite{anandkumar2014tensor,de2000best}, quantum processing, quantum geometry and data mining \cite{arrigo2019HITS,benson2019three,hu2016computing}. 
Even though computing a prescribed $Z$-eigenvector is in general NP-hard \cite{hillar2013most}, the use of higher-order ergodicity coefficients allows us to identify a class of nonnegative tensors for which the Perron eigenvector can be approximated efficiently to an arbitrary precision. While we focus here on stochastic tensors of order three, we believe  the results here presented can be further extended to more general eigenvector problems for nonnegative tensors.

The remainder of the paper is structured as follows:
We fix the relevant notation in the next section. In Section \ref{sec:HOMC} we review the concept of higher-order Markov chain, its associated $Z$-eigenvector stationary distribution and the issues related to the ergodicity of this type of higher-order stochastic process. In Section \ref{sec:main_res} we recall the concept of ergodicity coefficient for a stochastic matrix and some of its properties. Then, in Section \ref{sec:HO_erg_coeff}, we introduce our new higher-order ergodicity coefficients for stochastic cubic tensors
and we prove some of our main results. In Section \ref{sec:Z-eig} we show how these apply to the ergodicity of higher-order stochastic processes. In particular, after recalling the 
definition 
of vertex-reinforced and spacey random walks, we introduce in Subsection \ref{sec:spacey} a general family of Markov processes with memory that includes the spacey random walk as particular case and we prove a new convergence result for this general stochastic process.  In Section \ref{sec:previous-work} we compare the ergodicity coefficients introduced in Section \ref{sec:HO_erg_coeff} with analogous coefficients found in the recent literature. 
Finally, in Section \ref{sec:Z-eig-examples}, 
we discuss a number of application examples that showcase the advantages of the proposed results. In particular, we consider the computation of the multilinear PageRank and its application to triangle-based random walks in networks, and the convergence of the shifted higher-order power method.

\section{Notation}

Let $e_i$ be the $i$-th canonical basis vector in $\R^n$
and let $\uno$ be the all-ones vector.
Define the sets $\SS = \{ x\in\R^n : x \geq 0, \|x\|_1 = 1\}$ and 
$\SZ = \{ x\in\R^n : \uno^Tx = 0, \|x\|_1 = 1\}$. 
A real cubical tensor $\b P$ of order $3$ (or, equivalently, with $3$ modes) is a
three-way array with real entries of size $n\times n\times n$. 
We denote by $\mathbb R^{[3,n]}$ the set of such tensors and use capital bold letters to denote its elements. The $(i,j,k)$-entry
of $\b P\in \mathbb R^{[3,n]}$ is denoted by $\b P_{ijk}$. 
Matrices are tensors with only $2$ modes and are denoted with standard capital letters.  

Given a tensor $\b P \in \mathbb R^{[3,n]}$,  
we write $\b P xy$ to denote the tensor-vector  multiplication over the second and third modes:  
$$
   (\b Pxy)_{i} = \sum_{j,k=1}^n \b P_{ijk} x_j y_k
$$
for $i=1,\dots,n$. 
Moreover, the product $\TP x$ denotes the matrix associated with the linear map $y\mapsto \TP xy$,
that is, 
\begin{equation}   \label{eq:collapse}
   (\TP x)_{ij} = \sum_{k=1}^n \TP_{ikj}x_k .
\end{equation}
With this notation, it holds $(\TP x)y = \TP xy$.
A $Z$-eigenvalue of a tensor $\TP \in \mathbb R^{[3,n]}$ is a real number $\lambda$ such that there exists a nonzero vector $x\in\R^n$ such that $\TP xx = \lambda x$. Such vector $x$ is a $Z$-eigenvector associated with $\lambda$, see  \cite{QiLuo}.

There are $6=3!$ possible transpositions of a tensor $\b P\in \mathbb R^{[3,n]}$, each corresponding to a different permutation $\pi$ of the set $\{1,2,3\}$. Using the notation proposed in \cite{ragnarsson2012block}, the transposed tensor corresponding to the permutation $\pi$ can be denoted by $\b P^{\left<\pi\right>}$, namely,
$
     ( \b P^{\left<\pi\right>})_{ijk} = \b P_{\pi(i),\pi(j),\pi(k)} .
$ 
As it will be of particular importance to us, we devote the special notation  $\b P^S$ to denote the tensor obtained by transposing the entries of  $\b P$ over the second and third modes, namely 
$$
   \b P^S = \b P^{\left<[132]\right>}, \qquad  
   (\b P^S)_{ijk} = \b P_{ikj} .
$$
Moreover, we say that a tensor $\TP$ is $S$-symmetric
whenever $\TP = \TP^S$.

All inequalities in this work are meant entry-wise. In particular, we write $\b P\geq 0$ (resp., $\b P>0$) to denote a tensor such that $\b P_{ijk}\geq 0$ (resp., $\b P_{ijk}>0$) for all indices $i,j,k=1,\dots,n$. A tensor $\b P \in \mathbb R^{[3,n]}$ is said to be column stochastic or simply stochastic, if $\b P\geq 0$ and its first mode entries all sum up to one, i.e., $\sum_{i=1}^n \b P_{ijk} =1$, $\forall j,k=1,\dots,n$. 
A tensor acting as the identity on the unit sphere $x^Tx = 1$
can be defined in the case of tensors with an even number of modes, see \cite{KoMa11}.
For tensors with three modes we define the following two left $\b E^L$ and right $\b E^R$  ``one-sided identity tensors'':
\begin{equation}   \label{eq:Etensors}
   \b E^L_{ijk} = \delta_{ij} \qquad \text{and} \qquad 
   \b E^R_{ijk} = \delta_{ik} , \qquad \text{for all } i,j,k=1,\dots,n.
\end{equation} 
Both $\b E^L$ and $\b E^R$ 
are stochastic tensors  and for all $x\in \SS$ and $v\in\mathbb R^n$
one has $\b E^L vx = v$ and $\b E^R xv = v$.
Indeed,
$
   (\b E^L vx)_i = \sum_{jk}\b E^L_{ijk}v_jx_k = 
   \sum_{jk}\delta_{ij}v_j x_k = 
   v_i\sum_{k}x_k = v_i 
$ 
and similarly for $\b E^R$.
Note that, letting $\b E = \alpha \b E^L + (1-\alpha) \b E^R$ for any $\alpha\in[0,1]$, it holds $\b E xx = x$ for all $x\in\SS$.


\section{Higher-order Markov chains} \label{sec:HOMC}

Higher-order Markov chains
are a natural extension of Markov chains, where the transition probabilities depend on the past
few states, rather than just the last one. 
For a plain introduction, see e.g., \cite{BGL17,WuChu17}.
For example, a discrete-time second-order Markov chain is defined by a third-order tensor  $\TP = (\TP_{ijk})$ where $\b P_{ijk}$ is the conditional probability of transitioning to state $i$, given that the last state was $j$ and the second last state was $k$.
More precisely, if $X(t)$ is the random variable describing the status of the chain
on the set $\{1,\ldots,n\}$ at time $t=0,1,\ldots$, then 
$$
   \TP_{ijk} = \mathbb{P}(X(t+1) = i
   | X(t) = j, X(t-1) = k) ,
$$
where $\mathbb{P}$ denotes probability.
Hence, the sequence $\{X(t)\}$ obeys the rule
\begin{equation}   \label{eq:Markov2}
   \mathbb{P}(X(t+1) = i) = \sum_{j,k}
   \TP_{ijk} \mathbb{P}(X(t) = j,X(t-1) = k) .
\end{equation}
Obviously it must hold $\sum_i \TP_{ijk} = 1$
for $j,k = 1,\ldots,n$, i.e., the tensor $\TP$ is stochastic.

Let $x_t\in \SS$ be the probability vector of the random variable $X(t)$, i.e., the vector with entries $(x_t)_i = \mathbb{P}(X(t) = i)$. 
Let $Y_t$ denote the joint probability function $(Y_t)_{ij}= \mathbb{P}( X(t) = i, X(t-1) = j) $. 
Then, $x_t$ is the marginal probability $Y_t \uno$, i.e., 
the vector with entries $(x_t)_i = \sum_{j}(Y_t)_{ij}$. 
Hence, the dynamic of the second-order Markov chain \eqref{eq:Markov2} is described by the two-phase process \vspace{-.5em}
\begin{align}\label{eq:2ndMC}
\begin{cases}
    (Y_{t+1})_{ij}=\sum_{k}\TP_{ijk}(Y_t)_{jk} &  \\
    (x_{t+1})_i = \sum_{j} (Y_{t+1})_{ij} . & 
    \end{cases}
\end{align}
Note that both steps in \eqref{eq:2ndMC} are linear and thus their convergence can be analyzed using standard ergodicity  arguments. In fact, the second-order Markov chain over the state set $\{1,\dots, n\}$ can be easily reduced to 
a first-order Markov chain with state space $\{1,\dots, n\}\times \{1,\dots, n\}$,
see e.g., \cite{BGL17,WuChu17}. 
Thus, 
under appropriate hypotheses on $\TP$, the iteration \eqref{eq:2ndMC}
has a unique limit  $Y\geq 0$ such that 
\begin{equation}   \label{eq:YPY}
   Y_{ij} = \sum_{k=1}^n \TP_{ijk}Y_{jk} .
\end{equation}
However, this approach has a computational drawback: the size of the joint probability function of a second-order Markov chain is
the square of the number of states. The 
situation gets even worse for an $m$-th order Markov chain due to the ``curse of dimensionality'' effect: the memory space required by the joint density
grows exponentially with the states space size, requiring $n^{m}$ entries.  Moreover, the convergence analysis of the iteration \eqref{eq:2ndMC} and its natural extension to the $m>2$ setting becomes cumbersome. 

In order to circumvent these issues, 
Raftery \cite{Raftery85} proposed a technique to approximate higher-order Markov chains by means of a linear combination of first-order ones, by assuming that the joint probability distribution of the lagged random variables 
$X(t),\ldots,X(t-m+1)$ can be replaced by a mixture of its marginals. 
In the second-oder ($m = 2$) case, that assumption reduces to replacing the conditional probabilities $\TP_{ijk}$ by an expression of the form
$\lambda Q_{ij} + (1-\lambda)Q_{ik}$, where $Q$ is a stochastic matrix and $\lambda\in[0,1]$.
This technique, known as the Mixture Transition Distribution model, has been widely used to 
fit stochastic models with far fewer parameters than the fully parameterized model
to multi-dimensional data in a variety of applications  \cite{BerRaf02,RafTav94}.

A more recent approach, which maintains all the information contained in the transition tensor $\TP$, is the one proposed by Li and Ng
in \cite{LiNg}. Here, 
still in the $m = 2$ case, one assumes that the joint probability distribution of the higher-order Markov chain is the  
tensor product of its marginal distributions, that is, $Y_t=x_t\,   x_{t-1}^T$.
This hypothesis, which is equivalent to assuming that the random variables $X(t)$ and $X(t-1)$ are independent, 
is a conceptual simplification of the Markov chain formalism that is introduced in order to obtain a computationally tractable extension to the  second-order case.
  The resulting process is the quadratic version of a nonlinear Markov process \cite{kolokoltsov2010nonlinear}  
and it is still called a second-order Markov chain by many authors, see e.g., \cite{HuQi14,LiZhang15,LiNg}.
In this work, we will follow this well established convention.

Using our tensor-vector product notation, this ``reduced'' higher-order Markov process boils down to
the iteration
\begin{equation}   \label{eq:2omc}
   x_{t+1} = \b P  x_t x_{t-1} ,
\end{equation}
which
replaces \eqref{eq:2ndMC} and is the higher-order counterpart of the usual Markov process for a stochastic matrix in the  classical (first-order) Markov chain setting. The limit of this sequence, if it exists, is a nonnegative vector $ x\in \mathcal S_1$ such that 
\begin{equation}   \label{eq:xPxx}
    x = \TP  x  x ,
\end{equation}
that is, $ x$ is a $Z$-eigenvector of $\b P$ associated with the $Z$-eigenvalue $1$. Thus, it is natural to consider 
any such
vector as a stationary density of the Markov chain \eqref{eq:Markov2}.

Note that the limit matrix $Y$  of \eqref{eq:2ndMC} is such that 
$Y\uno = Y^T\uno$. Indeed, from \eqref{eq:YPY} we have
$$
   (\uno^T Y)_j = \sum_i Y_{ij} = 
   \sum_k Y_{jk} \sum_{i}\TP_{ijk} = 
   \sum_k Y_{jk} = (Y\uno)_j .
$$
But that row-column sum is generally different from 
the vector $x$ in \eqref{eq:xPxx}.
In fact, that solution
corresponds to a case where $Y$ has rank one, 
namely, $Y =  x x^T$.
Indeed, if $Y$ in \eqref{eq:YPY} is such that $ \mathrm{rank}(Y)=1$, then $x = Y\uno$ must solve \eqref{eq:xPxx}. 

On the other hand, the converse implication is false in general; that is, if $x$ solves \eqref{eq:xPxx}
then the matrix $Y = x x^T$ may not be a solution of \eqref{eq:YPY}. 
Indeed, extensive numerical experiments reported in \cite{WuChu17} show that the vector $x$ is strongly correlated with the row-column sum vector of $Y$, but the matrix $Y$ has full rank in general and $x \neq Y\uno$.

\subsection{Ergodicity of higher-order Markov chains}

In the matrix case, a Mark\-ov chain is called ergodic whenever it has a unique stationary vector and, for any initial probability distribution, that vector is the limiting distribution of the chain. Necessary and sufficient conditions for ergodicity of Markov chains
are well known, and are essentially related to spectral and structural properties, e.g., irreducibility 
and aperiodicity, of the transition matrix
\cite{SenetaBook}.

The situation complicates significantly when moving from matrices to tensors and, more generally, from linear to nonlinear cases \cite{kolokoltsov2010nonlinear}. In fact, even though  the existence of a solution $x\in \SS$ to \eqref{eq:xPxx} is a direct consequence of the Brouwer's fixed point theorem, the properties that characterize uniqueness and convergence of the process to the stationary distribution do not extend straightforwardly from the matrix case.  For example, unlike the matrix case, the irreducibility of $\b P$ is not enough to ensure the uniqueness of $x$ and additional assumptions are required. 
In fact, it is not too difficult to produce examples of entrywise positive stochastic tensors for which the equation \eqref{eq:xPxx} has multiple solutions,  or the solution of \eqref{eq:xPxx} is unique
but the iteration \eqref{eq:2omc} fails to converge to that solution. For instance,  a $\TP\in\R^{[4,2]}$ example is provided by Chang and Zhang in \cite{chang2013uniqueness}, while several $\TP\in\R^{[3,3]}$ examples are provided by Saburov in \cite{Sab19}. 

A  sufficient condition that ensures ergodicity is the existence of a metric 
with respect to which the system is contractive. Even though, as in the linear case, this is a sufficient but not a necessary requirement in general, 
suitable choices of the  metric  can provide valuable  
  conditions for the ergodicity of higher-order stochastic processes that can be given in terms of the entries of the tensor $\b P$. 
  
By considering the $\ell^1$ and the Hilbert metrics on $\SS$, in the following we introduce a family of ergodicity coefficients for stochastic cubic tensors of order three and we  show, in Section \ref{sec:Z-eig}, how they allow us to prove new conditions for the ergodicity of various higher-order stochastic processes. 
The conditions we obtain in this way can be easily computed and are, to the best of our knowledge, among the weakest conditions available in the literature so far.

\section{Coefficients of ergoditicy}\label{sec:main_res}

Let $d:\SS\times \SS\to \mathbb R_+$ be a metric on $\mathcal S_1=\{x\geq 0 : x^T\uno = 1\}$ 
and  consider a mapping  $f:\mathcal S_1\to\SS$. Although other notions of 
ergodicity coefficient are available in the literature, see e.g., \cite{Ergodic}, 
for the purpose of this work a coefficient of ergodicity for $f$ 
is the best Lipschitz constant of $f$ with respect to $d$, that is
\begin{equation}\label{eq:tau(f)}
   \tau_d(f) = 
   \sup_{\substack {{x,y\in \mathcal S_1} \\
   {x\neq y}}}\frac{d(f(x),f(y))}{d(x,y)} \, .
\end{equation}
Different choices of the metric $d$ give rise to different notions of ergodicity coefficients. For example, if $d$ is the Hilbert projective distance 
\begin{equation}\label{eq:hilbert}
     d_H(x,y) =
   \log \Big(\max_i\frac{x_i}{y_i} \max_i\frac{y_i}{x_i}\Big)
\end{equation}
then \eqref{eq:tau(f)} is the so-called Birkhoff contraction ratio \cite{birkhoff1}, which we denote by $\tau_H(f)$. This choice of metric is particularly interesting because it extends very naturally to the case of mappings $f$ that leave the slice of a generic proper cone  invariant (not just $\mathcal S_1$). Moreover, when $f$ is a linear map described by the matrix $A$, the Birkhoff--Hopf theorem \cite{BHNB} provides an explicit formula for $\tau_H(f)=\tau_H(A)$, which we recall below:

$$
   \tau_H(A) = \tanh\left(\frac14 \log \Big(\max_{ijhk}
   \frac{A_{ij}A_{hk}}{A_{ik}A_{hj}}\Big) \right) ,
$$
where  $\tanh(\lambda) = (e^\lambda-e^{-\lambda})/(e^\lambda+e^{-\lambda})$ denotes the hyperbolic tangent. 
An equivalent formula can be found also in \cite[\S 3.4]{SenetaBook}.
More recently, in \cite{gautier2018contractivity}, an analogous explicit formula has been proved for the case where $f$ is a (weakly) multilinear mapping induced by a nonnegative tensor. In particular, this formula holds for the case of $Z$-eigenvectors of cubic stochastic tensors and we will review it in this setting in Subsection~\ref{sec:BH}.

Another popular and successful choice for the distance in \eqref{eq:tau(f)} is $d(x,y) = \|x-y\|_p$, where $\|\cdot\|_p$ is the $p$-norm on $\mathbb R^n$.  Norm-based coefficients were introduced by Dobrushin in 1956 \cite{dobrushin1956central}  for the case of linear mappings and have been the subject of numerous investigations afterwards, see e.g., \cite{Ergodic,Sab19,SenetaBook}. 

In Section \ref{sec:HO_erg_coeff} we  analyze properties of norm-based coefficients for mappings defined by a stochastic tensor $\b P$.
To this end, we first review some relevant properties of these coefficients for the case of linear maps. 

\subsection{Norm-based ergodicity coefficients for matrices}
\label{par:erg_coeff_matrices}

Let $P$ be a stochastic matrix and $p\geq 1$. The $p$-norm ergodic coefficient of $P$ is 
$$
   \tau_p(P) = \sup_{\substack {{x,y\in \mathcal S_1}\\{x\neq y}}}\frac{\|Px-Py\|_p}{\|x-y\|_p} .
$$
This definition extends obviously to any matrix $P\in\R^{n\times n}$, when appropriate.
The linearity of $P$, the continuity of $\|\cdot\|_p$ and the fact that the set
$\{z \in \mathbb R^n : z=(x-y)/\|x-y\|_p,\,  x,y\in \mathcal S_1\}$ 
coincides with $\mathcal{Z}_p = \{z \in \R^n : \|z\|_p = 1,\, \uno^Tz = 0\}$, which is compact,
yield the equivalent formula
$$
  \tau_p(P) = 
  \max_{\substack {{\|x\|_p=1}\\{x^T \uno = 0}}} \|Px\|_p  .
$$
Norm-based ergodicity coefficients for stochastic  matrices $P$ are particularly useful for three reasons:  they provide sufficient conditions for the ergodicity of the Markov chain  associated with $P$;
they can be used to derive  bounds on the variation of the stationary distribution of the Markov chain, when the transition probabilities face a small perturbation;  
and,  
in the case $p = 1$  they yield easily computable upper bounds on the convergence rate of the Markov process $x_{t+1}=Px_t$. We review these properties in the next Theorems \ref{thm:pert}, \ref{thm:ergodic_thm_matrices} and  \ref{thm:formulas_for_tau_matrices}. Then, in Sections \ref{sec:HO_erg_coeff} and \ref{sec:Z-eig}, we will use the $1$-norm ergodicity coefficients for matrices to introduce what we call higher-order ergodicity coefficients for stochastic tensors of order three and we will show that the above three fundamental properties carry over to the tensor case.
We refer to \cite{Ergodic,SenetaBook,tudisco2015note} for 
more details on $\tau_p(P)$.

The following properties follow directly from the definition of $p$-norm ergodic coefficient.

\begin{theorem}\label{thm:tau_p_properties}
If $P,Q\in\R^{n\times n}$ are stochastic, then
\begin{enumerate}
\item $0\leq \tau_p(P)\leq \|P\|_p$
\item $|\tau_p(P)-\tau_p(Q)|\leq \tau_p(P-Q)$
\item $\tau_p(P) = 0$ if and only if $\mathrm{rank}(P)=1$.
\end{enumerate}
\end{theorem}

Moreover,
the following perturbation bound holds (see e.g.\ \cite{Seneta88} or \cite[Thm.\ 3.14]{Ergodic}).

\begin{theorem}    \label{thm:pert}
Let $P,P'$ be two stochastic irreducible matrices, and let $x,x'$ be their corresponding stationary probability vectors.
Then
$$
   \|x - x'\|_p \leq \frac{\|P-P'\|_p}{1-\tau_p(P)} .
$$
\end{theorem}

As an immediate consequence of the definition \eqref{eq:tau(f)}, the inequality $\tau_p(P) < 1$ implies that the map $f:\SS\mapsto\SS$ defined by $f(x) = Px$ is a contraction.
This observation implies the following result.

\begin{theorem}   \label{thm:ergodic_thm_matrices} 
If $P$ is a stochastic matrix with $\tau_p(P)<1$ for some $p\geq 1$ then $P$ is ergodic, i.e., there exists a unique eigenvector $x\in \mathcal S_1$ such that $Px=x$.
Moreover, the power method $x_{t+1}=Px_{t}$ converges to $x$ for any $x_0\in\mathcal S_1$, and 
$$
   \| x_t - x \|_p \leq \tau_p(P)^t \| x_0 - x \|_p .
$$
\end{theorem}
The theorem above gives a sufficient condition for the ergodicity of $P$ which is very useful in practice  when combined with a number of explicit formulas that allow to compute $\tau_p(P)$ using only the entries of $P$. 
Here we recall those for the particular case $p=1$ \cite{dobrushin1956central}. 

\begin{theorem}\label{thm:formulas_for_tau_matrices}
Let $P\in\R^{n\times n}$. Then
$$
   \tau_1(P) = \frac12 \max_{j,k} \sum_i |P_{ij} - P_{ik}| .
$$
Moreover, if $P$ is stochastic then
\begin{align*}
\tau_1(P) &= 1 - \min_{jk}\sum_{i} \min\{P_{ij},P_{ik}\} = 1 - \min_{\II\subset\{1,\ldots,n\}} \bigg(\min_{j}\sum_{i\notin\II}P_{ij}+\min_{k}\sum_{i\in\II}P_{ik}\bigg) .
\end{align*}
\end{theorem}

We will devote Sections \ref{sec:HO_erg_coeff} and \ref{sec:Z-eig} to extend the ergodicity coefficient $\tau_1(P)$ 
to three-mode tensors, to prove analogous theorems to the preceding ones and to discuss further properties and applications.

\section{Ergodicity coefficients for third-order tensors}   
\label{sec:HO_erg_coeff}

Let $\b P \in \R^{[3,n]}$ be a cubic stochastic tensor. 
We define the following higher-order ergodicity coefficients: 
\begin{align}
\begin{aligned}\label{eq:TauDef}
   \Tau_L(\TP) & = \max_{x\in\SS} \max_{y\in\SZ} \|\b Pxy\|_1 \\
   \Tau_R(\TP) & = \max_{x\in\SS} \max_{y\in\SZ} \|\b Pyx\|_1 \\
   \Tau(\TP) & = \max_{x\in\SS} \max_{y\in\SZ} \|\b Pxy + \b Pyx\|_1 .
\end{aligned}
\end{align}  
The preceding definitions are extended obviously to any tensor $\b P \in \R^{[3,n]}$,
when appropriate. 
We remark the following immediate identities:
$$
   \Tau_L(\TP) = \Tau_R(\TP^S) , \qquad 
   \Tau(\TP) = 2 \, \Tau_L(\textstyle \frac12 \TP + \frac12 \TP^S) 
   = 2 \, \Tau_R(\textstyle \frac12 \TP + \frac12 \TP^S).
$$
In particular, for an $S$-symmetric tensor $\TP$ we have
$\Tau_L(\TP) = \Tau_R(\TP) = \frac12\Tau(\TP)$.

The relationship between the preceding definitions and the norm-based ergodicity coefficients considered in Section  \ref{par:erg_coeff_matrices} 
can be 
revealed by considering the matrices associated with the tensor-vector
products $\TP x$ and $\TP^S x$ defined as in \eqref{eq:collapse}. In fact, it is not difficult to see that the following identities hold,
\begin{align*}
   \Tau_L(\TP)  = \max_{x\in\SS}\tau_1(\TP x),  \qquad
   \Tau_R(\TP)  = \max_{x\in\SS}\tau_1(\TP^S x) , \qquad
   \Tau(\TP)  =  \max_{x\in\SS}\tau_1( \TP x + \TP^S x) \, . 
\end{align*}
The above formulas yield a characterization of the three coefficients in \eqref{eq:TauDef} which, for example,  was used in \cite{Sab19} to define $\Tau(\TP)$ in the case of $S$-symmetric tensors.
Hereafter, we exploit these formulas to derive 
explicit expressions for computing the coefficients above 
from the knowledge of the tensor entries 
 and provide the tensor equivalent of Theorem \ref{thm:formulas_for_tau_matrices}.

\begin{theorem}   \label{thm:TauL}
Let $\b P \in \R^{[3,n]}$. Then,
\begin{equation}   \label{eq:TauL}
   \Tau_L(\b P) = \frac 1 2 \max_{j,k_1,k_2}\sum_{i} |\b P_{ijk_1}-\b P_{ijk_2}|  .  
\end{equation}
Moreover, if $\b P$ is stochastic then
\begin{align}
   \Tau_L(\b P) & = 1-\min_{j,k_1,k_2}\sum_{i} 
   \min\{ \b P_{ijk_1},\b P_{ijk_2}\}  \label{eq:TauLs}   \\
   & = 1 - \min_{\II\subset[n]} \min_j \Big(
   \min_{k_1} \sum_{i\notin \II}\b P_{ijk_1} + 
   \min_{k_2} \sum_{i\in \II} \b P_{ijk_2} \Big)  \label{eq:TauLs2} . 
\end{align}
\end{theorem}

\begin{proof}
For $i = 1,\ldots,n$ let $P^{(i)}$ be the stochastic matrix $P^{(i)}_{jk} = \TP_{jik}$.
Hence, for any $x\in\SS$ and $y\in\SZ$ we have $\TP xy = \sum_i x_iP^{(i)}y$.
By the triangle inequality,
$$
   \|\TP xy \|_1 =
   \Big\| \sum_i x_i P^{(i)}y \Big\|_1 \leq
   \sum_i x_i \| P^{(i)}y \|_1 \leq
   \max_i \tau_1 (P^{(i)}) .
$$
On the other hand, if $x = e_i$ where $i$ is an index such that $\tau_1 (P^{(i)}) \geq \tau_1 (P^{(j)})$
for $j = 1,\ldots,n$ and $y\in\SZ$ is a vector such that $\| P^{(i)}y \|_1 = \tau_1 (P^{(i)})$ then the inequalities above are actually equalities. Thus
all the claims follow at once from Theorem
\ref{thm:formulas_for_tau_matrices}.
\end{proof}

The analogous formulas for the other higher-order coefficients in \eqref{eq:TauDef} are derived hereafter.

\begin{corollary}   \label{cor:TauR}
Let $\b P \in \R^{[3,n]}$. The following identities hold:
\begin{align}
   \Tau_R(\TP) & = 
   \frac12 \max_{j_1,j_2,k}\sum_{i=1}^n |\TP_{ij_1k} - \TP_{ij_2k}|   \label{eq:TauR} \\
   \Tau(\TP) & = 
   \frac12 \max_{j,k_1,k_2} \sum_i |\TP_{ijk_1} - \TP_{ijk_2} + \TP_{ik_1j} - \TP_{ik_2j}|     
    \label{eq:Tau}  .  
\end{align}
Moreover, if $\TP$ is stochastic then
\begin{align}
   \Tau_R(\TP) & = 1 -
   \min_{j_1,j_2,k}\sum_{i=1}^n \min\{ \TP_{ij_1k},\TP_{ij_2k}\}  \label{eq:TauRs}  \\
   & = 1 - \min_{\II\subset[n]} \min_{j_1,j_2,k} \Big( 
   \sum_{i\notin \II}\TP_{ij_1k}+ \sum_{i\in \II}\TP_{ij_2k}\Big)  
   \label{eq:TauRs2}  \\
   \Tau(\TP) & = 2 - \min_{j,k_1,k_2} \sum_i 
   \min\{\TP_{ijk_1} + \TP_{ik_1j},
   \TP_{ijk_2} + \TP_{ik_2j} \}   \label{eq:Taus}   \\
   & = 2 - \min_{\II\subset[n]} \min_j \Big(
   \min_{k_1} \sum_{i\in\II} (\TP_{ijk_1} + \TP_{ik_1j} )
   + \min_{k_2} \sum_{i\notin\II} (\TP_{ijk_2} 
   + \TP_{ik_2j}) \Big) .   \label{eq:Taus2}
\end{align}
\end{corollary}

\begin{proof}
Equations \eqref{eq:TauR}, \eqref{eq:TauRs} and \eqref{eq:TauRs2}
derive from the identity 
$\Tau_R(\TP) = \Tau_L(\TP^S)$ and equations \eqref{eq:TauL}, \eqref{eq:TauLs} and \eqref{eq:TauLs2}, respectively.
Now, define $\b Q = \frac12(\TP+\TP^S)$. 
Note that if $\TP$ is stochastic then also $\b Q$ is stochastic.
Since
$$
   \|\b P xy + \b P yx\|_1 = \| \b P xy + \b P^S xy\|_1 = 
   2\| \textstyle \frac12 (\b P + \b P^S) xy\|_1 = 
   2\| \b Q xy\|_1 ,
$$
we have $\Tau(\b P) = 2\Tau_L(\b Q)$.
Hence, equations \eqref{eq:Tau}, \eqref{eq:Taus} and \eqref{eq:Taus2}
derive from the latter identity 
and equations \eqref{eq:TauL}, \eqref{eq:TauLs} and \eqref{eq:TauLs2}, respectively.
\end{proof}

By \eqref{eq:TauLs}, \eqref{eq:TauRs} and \eqref{eq:Taus}, it is immediate to observe that for a stochastic tensor $\TP$ it holds
$0 \leq \Tau_L(\TP),\Tau_R(\TP)\leq 1$ and 
\begin{equation}   \label{eq:TauLR}
   0 \leq \Tau(\TP) \leq \Tau_L(\TP)+\Tau_R(\TP) \leq 2 .
\end{equation}
Stronger inequalities can be easily obtained for positive tensors,
as shown in the next result.

\begin{corollary}   \label{cor:trivial}
Let $\TP\in\R^{[3,n]}$ be a stochastic tensor. If there exists a positive number $\alpha > 0$ such that 
$\TP_{ijk} \geq \alpha$ for all $i,j,k$ then
$$
   \Tau_L(\TP) \leq 1-n\alpha , \qquad
   \Tau_L(\TP) \leq 1-n\alpha , \qquad 
   \Tau(\TP) \leq 2(1-n\alpha) .
$$
\end{corollary}

\begin{proof}
The three inequalities in the claim follow immediately
from equations \eqref{eq:TauLs}, \eqref{eq:TauRs} and \eqref{eq:Taus},
respectively.
\end{proof}

\begin{remark}   \label{rem:rk1}
A close look at Theorem \ref{thm:TauL}
reveals that, for any tensor $\TP\in\R^{[3,n]}$ we have 
$\Tau_L(\TP) = 0$ if and only if $\TP_{ijk} = A_{ij}$ for some matrix $A\in\R^{n\times n}$.
In particular, $\TP$ is stochastic if and only if $A$ is stochastic.
Analogously, from Corollary \ref{cor:TauR}
we derive that $\Tau_R(\TP) = 0$ if and only if $\TP_{ijk} = A_{ik}$ for some matrix $A$.
Consequently, $\Tau_L(\TP) +\Tau_R(\TP) = 0$ if and only if $\TP_{ijk} = v_{i}$ 
for some vector $v$. 
It is not difficult to prove that the latter is also equivalent to $\Tau(\TP) = 0$.
Hence, if $\TP$ is nonzero, we have 
$$
   \Tau_L(\TP) +\Tau_R(\TP) = 0 \quad \Longleftrightarrow \quad 
   \Tau(\TP) = 0 \quad \Longleftrightarrow \quad 
   \mathrm{rank}(\TP) = 1 .
$$
In the matrix case, a coefficient of ergodicity $\tau$
is called 
proper when the identity $\tau(P) = 0$ for a stochastic matrix $P$ is equivalent to the condition $\mathrm{rank}(P) = 1$,
see \cite{Ergodic,SenetaBook}. For example, both the 
Birkhoff coefficient $\tau_H$
and all the norm-based ergodicity coefficients $\tau_p$ are proper. By extending that definition to the tensor case, we can say that $\Tau$ is proper,
while $\Tau_L$ and $\Tau_R$ are not proper. 
\end{remark}

The remark above shows that Property 3 of Theorem \ref{thm:tau_p_properties} carries over to the higher-order setting. In the next Subsection \ref{sec:bounding} we show that also Properties 1 and 2 of that theorem enjoy a tensor counterpart. In Subsection \ref{sec:perturbation}, instead, we show how the perturbation result of Theorem \ref{thm:pert} transfers to stochastic tensors.

\subsection{Bounding the variation of higher-order  coefficients}\label{sec:bounding}
When working with stochastic tensors, it is quite natural to endow $\R^{[3,n]}$ with the norm
$$
   \|\TP\|_1  = 
   \max_{\|x\|_1 = \|y\|_1 = 1}\|\TP xy\|_1 =  \max_{j,k} \sum_i |\TP_{ijk}|\, ,
$$
so that, if $\TP$ is stochastic, we have $\|\TP\|_1 = 1$. 

With the next theorem we prove a Lipschitz-continuity condition for the higher-order ergodicity coefficients with respect to the tensor $1$-norm above.
\begin{theorem}
For arbitrary $\TP, \b Q\in\R^{[3,n]}$ we have
$$
   |\Tau_*(\TP) - \Tau_*(\b Q)| \leq \Tau_*(\TP - \b Q) \leq \|\TP - \b Q\|_1 
$$
where $\Tau_*$ is any of $\Tau_L$ or $\Tau_R$. Moreover,
$$
   |\Tau(\TP) - \Tau(\b Q)| \leq \Tau(\TP - \b Q) \leq 2 \|\TP - \b Q\|_1 .
$$
\end{theorem}

\begin{proof}
Let $\Tau_* = \Tau_L$, the other case being completely analogous. Suppose
that $\Tau_L(\TP) \geq \Tau_L(\b Q)$. Hence,
for some $x\in\SS$ and $y\in\SZ$ we have
$$
   \Tau_L(\TP) = \| \TP xy \|_1 \leq 
   \| (\TP - \b Q)  xy \|_1 + \| \b Q xy \|_1
   \leq 
   \Tau_L(\TP - \b Q) + \Tau_L(\b Q) .
$$
Hence, $\Tau_L(\TP) - \Tau_L(\b Q) \leq \Tau_L(\TP - \b Q)$. 
By reversing the roles of $\TP$ and $\b Q$ we obtain 
$\Tau_L(\b Q) - \Tau_L(\TP) \leq \Tau_L(\TP - \b Q)$ 
and we arrive at the first claim. 
Analogously, for some $x\in\SS$ and $y\in\SZ$ we have
\begin{align*}
   \Tau(\TP) = \| \TP xy + \TP yx \|_1 & \leq 
   \| (\TP - \b Q)  xy + (\TP - \b Q)  yx \|_1 + \| \b Q xy + \b Q yx \|_1  \\
   & \leq 
   \Tau(\TP - \b Q) + \Tau(\b Q) .
\end{align*}
The inequality $\Tau(\b Q) - \Tau(\TP) \leq \Tau(\TP - \b Q)$ follows from the preceding one 
by exchanging $\TP$ and $\b Q$, and the second claim follows.
The rightmost inequalities follow immediately from the definition 
of the ergodicity coefficients.
\end{proof}


\section{Second-order stochastic processes and $Z$-eigenvectors}    
\label{sec:Z-eig}

In this section we prove an analogous of Theorem 
\ref{thm:ergodic_thm_matrices}
for tensor $Z$-eigenvectors. 
Precisely, given $\b P$ stochastic, we provide a new condition that ensures the existence and uniqueness  of a positive vector $x\in\SS$ such that $x=\b P xx$. 
Moreover, we show that under the same condition the 
higher-order power method
$x_{t+1} = \b P x_t x_t$, which is the prototypical 
nonlinear Markov process
\cite{kolokoltsov2010nonlinear},
always converges to $x$ and we provide an analogous, but stronger, condition that guarantees the global convergence of the alternate scheme $x_{t+1} = \b P x_t x_{t-1}$.

The next theorem provides the tensor analogous of Theorem \ref{thm:ergodic_thm_matrices}.

\begin{theorem}   \label{thm:ergodic_thm_tensor} 
If $\TP$ is  stochastic, then  $\Tau(\TP)$ is the best Lipschitz constant of the quadratic map $f:\SS\mapsto\SS$ given by $f(x) = \TP xx$, that is,
$$
   \Tau(\TP) =\tau_1(
   f) =
   \sup_{x,y\in \SS}\frac{\|\TP xx-\TP yy\|_1}{\|x-y\|_1}\, .
$$
Therefore, if  $\Tau(\TP)<1$ then  
there exists a unique $Z$-eigenvector $x\in \SS$ such  that $\TP xx=x$. Moreover, the higher-order power method  $x_{t+1}=\TP x_tx_t$  converges to $x$ for any $x_0\in\SS$, and
$$
   \|x_t - x \|_1 \leq \Tau(\TP)^t \|x_0 - x \|_1 .
$$
\end{theorem}

\begin{proof}
Let $f:\SS\mapsto\SS$ be given by $f(x) = \TP xx$.
Let $\b Q = \frac12(\TP + \TP^S)$. Note that $\b Q$ is a stochastic tensor such that $\b Q = \b Q^S$. Moreover, 
the equation $f(x) = x$ is equivalent to $\b Qxx = x$.
Then, for all $x,y\in\SS$ we have
\begin{align*}
   f(x) - f(y) & = \b Q xx - \b Q yy + \b Q xy - \b Q xy \\
   & = \b Q xx - \b Q yy + \b Q yx - \b Q xy 
   = \b Q (x + y)(x - y) .
\end{align*}
Hence, 
\begin{align*}
   \tau_1(f) & = \max_{x,y\in\SS} \frac{\|f(x) - f(y)\|_1}{\|x-y\|_1}  = \max_{x,y\in\SS} \frac{ 2 \| \b Q (\frac12 x + \frac12 y)(x - y) \|_1}{\| x - y \|_1}   \\
   & = \max_{v\in\SS} \max_{w\in\SZ} 2 \|\b Q vw \|_1 
   =  2\, \Tau_L(\b Q) .
\end{align*}
Since $2\Tau_L(\b Q) = \Tau(\TP)$, 
we obtain the first part of the claim. In particular, we get
$\|f(x) - f(y)\|_1 \leq \Tau(\TP) \|x-y\|_1$ for any $x,y\in\SS$.
Hence, if $\Tau(\TP) < 1$ then
$f$ is contractive with respect to the $1$-norm. By the Banach fixed point theorem, there exists a unique fixed point $x\in\SS$ such that
$x = f(x)$. Moreover, the iteration $x_{t+1} = f(x_t)$ converges to $x$
with $\|x_t - x \|_1 \leq \Tau(\TP)\|x_{t-1} - x \|$ 
for any $x_0\in\SS$ and 
the proof is complete.
\end{proof}

We note in passing
that the following result, which has been derived from a well-known uniqueness result in the fixed point theory
several times by different authors
\cite{cui2019uniqueness,li2017uniqueness,LiNg}, is a direct consequence of the theorem above.

\begin{corollary}   \label{cor:chinese}
If $\TP\in\R^{[3,n]}$ is a stochastic tensor such that
$\TP_{ijk} > 1/(2n)$ for all $i,j,k$, then 
there exists a unique $Z$-eigenvector $x\in\SS$ such that $\TP xx = x$ and the higher-order power method $x_{t+1} = \TP x_t x_t$ converges to $x$ for any $x_0 \in\SS$.
\end{corollary}

\begin{proof}
In the stated hypotheses we have 
$\Tau(\TP) < 1$ 
by virtue of Corollary \ref{cor:trivial}.
Hence, the claim is a direct consequence of Theorem \ref{thm:ergodic_thm_tensor}.
\end{proof}

Given a stochastic tensor $\b P$ and two initial points $x_0,x_{-1}\in \SS$, the following alternate higher-order power method has been considered 
in \cite{HuQi14}:
\begin{equation}\label{eq:PM2}
    x_{t+1} = \TP x_t x_{t-1}, \qquad t=0,1,2,\dots
\end{equation}
Note that this coincides with the second-order stochastic process described in \eqref{eq:2omc}.
In \cite{HuQi14} the convergence of \eqref{eq:PM2}
has been proven when
$\TP_{ijk} > 1/(2n)$ and under restrictive hypotheses on the choice of $x_0$ and $x_{-1}$.
The following theorem provides a condition in terms of $\Tau_L(\b P)$ and $\Tau_R(\b P)$ that  ensures that \eqref{eq:PM2} converges globally to the unique fixed point of $\b P$.

\begin{theorem}\label{thm:PM2}
Let $\b P$ be a stochastic tensor
and let $s = \Tau_L(\b P) + \Tau_R(\b P)$. If 
$s < 1$ 
then the iteration \eqref{eq:PM2} converges to the unique
$Z$-eigenvector $x\in \SS$ such that $x=\b Pxx$.
In fact, for all $x_0,x_{-1}\in \SS$ 
and $t = 0,1,\ldots$ it holds
$$
   \|x_{t+1}-x\|_1 \leq s^{\lceil (t+1)/2 \rceil} 
   \max \{ \|x_0 - x\|_1 , \|x_{-1} - x\|_1 \}.
$$
\end{theorem} 

\begin{proof}
First notice that the assumption $\Tau_L(\b P) + \Tau_R(\b P)<1$ implies $\Tau(\b P)<1$, thus, by Theorem \ref{thm:ergodic_thm_tensor}, there exists a unique positive $x\in \SS$ such that $x=\b Pxx$. We have
\begin{align*}
   x_{t+1}-x = \b Px_t x_{t-1}-\b Pxx 
   & = \b Px_t x_{t-1}-\b Px_t x+\b Px_t x-\b Pxx \\
   & = \b Px_t (x_{t-1}-x)+\b P(x_t-x)x .
\end{align*} 
Thus, for any $t \geq 0$, 
\begin{align*}
   \|x_{t+1}-x\|_1 &\leq \Tau_L(\b P)\|x_{t-1}-x\|_1+\Tau_R(\b P) \|x_t-x\|_1 \\
   &\leq \big(\Tau_L(\b P) + \Tau_R(\b P)\big) \max\{\|x_{t}-x\|_1,\|x_{t-1}-x\|_1 \} .
\end{align*}
In particular, the claim is true for $t=0$.
The proof is completed by a simple inductive argument.
Indeed, 
let $m = \max \{ \|x_0 - x\|_1 , \|x_{-1} - x\|_1 \}$
and $\epsilon_t = \|x_t-x\|_1$
to simplify notations. 
For $t > 0$, suppose the claim true up to $t-1$. Then,
$$
   \epsilon_{t+1} \leq s\max\{\epsilon_t ,\epsilon_{t-1}\}
   \leq s\, \max\{s^{\lceil t/2 \rceil}, s^{\lceil (t-1)/2 \rceil}\}\, m
   =  s^{\lceil (t+1)/2 \rceil}\,  m ,
$$
and the theorem is proved.
\end{proof}

\subsection{A perturbation result for the stochastic $Z$-eigenvector} \label{sec:perturbation}

A fundamental perturbation analysis problem is to obtain quality bounds on the variation of the ergodic distribution of the  non-negative stochastic tensor $\TP$, when $\TP$ is perturbed.
The following result provides a bound in terms of the higher-order norm-based ergodicity coefficients, and represents a tensor counterpart of Theorem \ref{thm:pert}.

\begin{theorem}   \label{thm:cond3}
Let $\TP$ and its perturbation $\TP'$ be two stochastic tensors in $\R^{[3,m]}$.
If $\Tau(\TP) < 1$ then the stochastic solution of $x = \TP xx$ is unique, and for 
any stochastic vector $x'$ such that $x' = \TP' x'x'$ it holds
$$
   \|x-x'\|_1 \leq 
   \frac{\|\TP - \TP'\|_1}{1 - \Tau(\TP)} .
$$
\end{theorem}

\begin{proof}
Suppose first that both $\TP$ and $\TP'$ are $S$-symmetric. By adding and subtracting $\TP x'x'$ we have
\begin{align*}
    \|x - x'\|_1 & = \|\TP xx - \TP' x'x' + \TP x'x' - \TP x'x'\|_1 \\
    & \leq \|\TP x(x - x') + \TP (x-x')x'\|_1 
      + \|(\TP - \TP') x'x'\|_1 \\
    & = 2 \| \TP (\textstyle \frac12 x + \frac12 x')(x - x')\|_1 
      + \|(\TP - \TP') x'x'\|_1 \\
    & \leq 2\Tau_L(\TP) \| x - x' \|_1 + \|\TP - \TP'\|_1  = \Tau(\TP) \| x - x' \|_1 + \|\TP - \TP'\|_1 .
\end{align*} 
Rearranging terms we find $\|x-x'\|_1(1-\Tau(\TP))\leq \|\TP - \TP'\|_1$ and the claim follows.

In the general case, define $\b Q = \frac12(\TP+\TP^S)$ and $\b Q' = \frac12(\TP'+\TP'^S)$
and repeat the previous arguments. Finally, note that $\Tau(\b Q) = \Tau(\TP)$ and $\|\b Q - \b Q'\|_1 \leq \|\TP - \TP'\|_1$.
\end{proof}

\subsection{Convergence of a class of vertex reinforced random walks}\label{sec:spacey}

Vertex reinforced random walks are another important example of higher-order discrete-time stochastic process $\{X(t)\}_t$ on the state space $\{1, \dots, n\}$,  
where the state transitions at time $t$ depend on the whole 
history $X(0),\ldots,X(t-1)$ \cite{Ben97,Pemantle92}. 
Starting from an initial state $X(0) \in\{1,\ldots,n\}$, 
the process evolves according to the formulas
\begin{align}   \label{eq:vertex_reinforced}
\begin{aligned}
    \mathbb P(X(t+1)=i|\mathcal F_t) & = \mathcal M(y_t)_{i,X(t)} \\
     (y_t)_j & = \frac{1+\sum_{i=1}^t [X(i)=j]}{t+n} ,   
\end{aligned}
\end{align}
where $\mathcal F_t$ is the $\sigma$-field generated by $X(1),\dots, X(t)$, 
and $\mathcal M$ is a map from 
$\SS$ 
to the set of stochastic $n\times n$ matrices. 
The vector $y_t$, which is called the occupation vector,  
is an auxiliary stochastic vector that is introduced in order to record the history of the process.
Indeed, the $i$-th entry of $y_t$ is proportional to the number of times the process visited state $i$ up to the $t$-th time step, plus one. 
Now, let $x_t$ be the probability vector of $X(t)$, that is,
the $n$-vector whose $i$-th entry is $\mathbb{P}(X(t)=i)$.
Then, the process \eqref{eq:vertex_reinforced} can be equivalently described via the  coupled equations
$$
   x_{t+1} = \mathcal M(y_t)x_t , \qquad 
   y_{t} = \frac{1}{t+n} \sum_{s=1}^t x_s 
   + \frac{1}{t+n}\uno .
$$
When  $\mathcal M$ is linear, there exists a stochastic tensor $\TP$ such that $\mathcal M(v)_{ij} = \sum_k \TP_{ijk}v_k$ and  
the corresponding stochastic process is
the so-called spacey random walk, introduced in \cite{BGL17}. In this case,  with minor notation changes with respect to the original version, the previous iteration can be recast as
\begin{equation}   \label{eq:spacey}  
   \begin{cases}
   x_{t+1} = \TP x_ty_t  &  \\
   y_{t+1} = \frac{1}{t+1} x_t + \frac{t}{t+1} y_t .
   \end{cases}
\end{equation}
On the basis of key results by Bena\"im \cite{Ben97}, 
Benson, Gleich and Lim established the convergence of the spacey random walk
in terms of the convergence of a certain ordinary differential equation to a stable equilibrium,
and one auxiliary condition placed on $\TP$
\cite[Thm.\ 9]{BGL17}.
However, only the convergence of the 
occupation vectors $\{y_t\}$ 
(which corresponds to the convergence in the Ces\`aro average sense of the random variables $X(t)$)
can be derived from the results in \cite{Ben97,BGL17}.
In fact, the second equation in \eqref{eq:spacey} yields
$$
   y_{t+1} - y_t = \frac{1}{t+1} ( x_{t} - y_t  ) .
$$
Hence, even if the sequence $\{y_t\}$ has a limit and the left hand side converges to zero, that does not imply the convergence 
of the sequence $\{x_t\}$.

In what follows, we consider the following generalization
of \eqref{eq:spacey},
\begin{align}   \label{eq:vrrw}
\begin{cases}
  x_{t+1} = \TP x_ty_t & \\ 
  y_{t+1} = c_t x_t + (1-c_t) y_t  &
\end{cases}
\end{align}
with $c_t\in[0,1]$
and we show in the next theorem that, if the higher-order ergodicity coefficients $\Tau_L(\TP)$ and $\Tau_R(\TP)$ are small enough, then the stochastic process \eqref{eq:vrrw} is globally convergent, provided that the sequence $\{c_t\}$ is not too small. This requirement on $\{c_t\}$ can be seen as a  condition that avoids the process from freezing along the way on a 
limit point that is far away from the $Z$-eigenvector of $\TP$. 
In fact, the possibility of such a  behavior has been shown in \cite{bouguet2018fluctuations}
for a stochastic process closely related to \eqref{eq:vrrw}.

\begin{theorem}\label{thm:vrrw}
Let the sequence $\{c_t\}$ in \eqref{eq:vrrw} 
be non-increasing and such that
\begin{equation}   \label{eq:c_t}
   \sum_{t=1}^{\infty} c_t = +\infty .  
\end{equation}
If $\Tau_L(\TP) + \Tau_R(\TP) < 1$, then the vertex reinforced random walk \eqref{eq:vrrw} converges globally, i.e., for any starting points $x_0,y_0\in \SS$ we have 
$$
   \lim_{t\to\infty} \|x_t -  x\|_1 =  \lim_{t\to\infty} \|y_t -  x\|_1 = 0 ,
$$
where $x$ is the unique stochastic solution of $x = \TP  x x$.
Moreover, if there exists a positive constant $\alpha$ such that $c_t \geq \alpha$
then the convergence is linear.
\end{theorem}

\begin{proof}
Firstly, note that, in the stated hypotheses,
the vector $ x$ exists and is unique owing to Theorem \ref{thm:ergodic_thm_tensor}. 
Subtracting the identity $x = \TP x x$
from \eqref{eq:vrrw} we obtain 
\begin{align*}
   x_{t+1} -  x & = 
   \TP(x_t- x)y_{t} + \TP x(y_{t} - x) \\ 
   y_{t+1} -  x & = c_t (x_t- x) + (1-c_t) (y_{t} - x) .
\end{align*}
Let $\alpha_t = \|x_t- x\|_1$ and $\beta_t = \|y_t- x\|_1$. Using vector inequalities, we have
$$
   \begin{bmatrix} \alpha_{t+1} \\ \beta_{t+1} \end{bmatrix}
   \leq 
   \begin{bmatrix} \Tau_L(\TP) & \Tau_R(\TP) \\ c_t & 1-c_t \end{bmatrix}
   \begin{bmatrix} \alpha_t \\ \beta_t \end{bmatrix} .
$$
Let $\gamma_t = \|(\alpha_t, \beta_t)^T\|_\infty = \max\{\alpha_t, \beta_t\}$.
For notational simplicity, let $\ell = \Tau_L(\TP)$, $r = \Tau_R(\TP)$, and define
$$
   A_t = \begin{bmatrix} \ell & r \\ c_t & 1-c_t \end{bmatrix} .
$$
Hence, for $t = 1,2\ldots$ we have 
$$
   \gamma_{t+1} \leq \|A_t\cdots A_1 A_0 \|_\infty \gamma_0 .
$$
Moreover, since $\gamma_{t+1} \leq \|A_t\|_\infty \gamma_t$ and $\|A_t\|_\infty = 1$, we have $\gamma_{t+1} \leq \gamma_t$, 
that is, the sequence $\{\gamma_t\}$ is non-increasing.
Now, for $t \geq 1$ consider the product $A_{t}A_{t-1}$.
Simple computations show that
\begin{align*}
   A_{t}A_{t-1} & = \begin{bmatrix}
   \ell^2+rc_{t-1} &  r(\ell+1-c_{t-1}) \\ \ell c_t + (1-c_t)c_{t-1} & r c_t + (1-c_t)(1-c_{t-1})
   \end{bmatrix} \\
   \|A_{t}A_{t-1}\|_\infty 
   & = \max\{r + \ell(\ell+r), 1 - c_t(\ell+r) \} < 1 .
\end{align*}
In particular, if $\lim_{t\to \infty}c_t = 0$ 
then there exists an integer $t_*$ such that for $t\geq t_*$
it holds  $\|A_{2t}A_{2t-1}\|_\infty = 1 - c_{2t}(\ell+r)$.
Consequently, we have
\begin{align*}
   \gamma_{2t+1}  \leq \bigg( \prod_{j=t_*}^{t} \| A_{2j}A_{2j-1}\|_\infty \bigg) \|A_{2t_*-2}\cdots A_1 A_0\|_\infty \gamma_0  = C \prod_{j=t_*}^{t} (1 - c_{2j}(\ell+r)) ,
\end{align*}
where $C = \|A_{2t_*-2}\cdots A_1 A_0\|_\infty \gamma_0$.
In order to prove that $\lim_{t\to\infty}\gamma_t = 0$ it is sufficient to discuss the limit
$$
   \lim_{t\to\infty} \prod_{j=1}^{t} (1 - c_{2j}(\ell+r)) ,
$$
which exists and is nonnegative since all factors belong to $(0,1)$.
By a known result on the convergence of infinite products,
see e.g., \cite[p.\ 223]{Knopp},
the preceding limit is positive if and only if the series
$$
   \sum_{j=1}^{\infty} c_{2j}(\ell+r) 
$$
is convergent. Hence, if \eqref{eq:c_t} holds then
$\lim_{t\to \infty}\gamma_t = 0$ and we are done.

On the other hand, if $c_t \geq \alpha > 0$ 
then there exists a number $s \in(0,1)$
such that $\|A_{t}A_{t-1}\|_\infty \leq s$.
Hence,
$$
   \gamma_{2t} \leq \prod_{j=0}^{t-1} \|A_{2j+1}A_{2j}\|_\infty \gamma_0
   \leq s^{t} \gamma_0 ,
$$
and the last claim follows.
\end{proof}

Note that both the spacey random walk \eqref{eq:spacey} and the second-order Markov chain \eqref{eq:PM2} are particular cases of the stochastic processes \eqref{eq:vrrw}, corresponding to the choices $c_t = \frac{1}{t+1}$ and $c_t=1$, respectively. 
Observe that both these choices satisfy the assumption \eqref{eq:c_t}. Thus, the convergence condition for the second-order Markov chain of Theorem \ref{thm:PM2} also follows as a consequence of  Theorem \ref{thm:vrrw}. 
Moreover, we obtain  the following convergence result for the spacey random walk which, to the best of our knowledge, is the first   
result that gives explicit conditions that guarantee the convergence of both the occupation vector 
and the density distribution for this stochastic process.

\begin{corollary}
If $\Tau_L(\TP) + \Tau_R(\TP) < 1$ 
then the spacey random walk \eqref{eq:spacey} converges globally, i.e., for any starting points $x_0,y_0\in \SS$ we have 
$$
   \lim_{t\to\infty} \|x_t -  x\|_1 =  \lim_{t\to\infty} \|y_t -  x\|_1 = 0 ,
$$
where $x$ is the unique stochastic solution of $x = \TP x x$.
\end{corollary}

\begin{proof}
It suffices to observe that the coefficient sequence $\{c_t\}$ of the spacey random walk is a trailing sub-sequence of the harmonic sequence, hence the hypothesis
\eqref{eq:c_t} is fulfilled.
\end{proof}

\section{Comparison with previous works}\label{sec:previous-work}
In this section we discuss how the newly proposed higher-order ergodicity coefficient $\Tau(\TP)$,  based on the $1$-norm, compares with previous works. In particular,  we compare it with the contraction ratios proposed by Gautier and Tudisco in \cite{gautier2018contractivity}, where the Hilbert metric is used to quantify the contractivity of multilinear  operators, and with the coefficients  introduced by Li and Ng in \cite{LiNg} in order to characterize the uniqueness of stationary distributions of stochastic tensors.

\subsection{Higher-order Birkhoff coefficients}\label{sec:BH}

When $d$ is the Hilbert projective metric $d_H$ defined in \eqref{eq:hilbert}, the ergodicity coefficient \eqref{eq:tau(f)} is known as Birkhoff contraction ratio and the renowned Birkhoff--Hopf theorem provides an explicit formula for such coefficient 
when $f$ is a linear map.
Recently, the Birkhoff--Hopf theorem has been extended to the case of multilinear mappings \cite{gautier2018contractivity}. We review that theorem in the following, for the case of a bilinear map $f:\R^n\times \R^n \to\R^n$ described by a cubic tensor $\TP$ as $f(x,y) = \TP xy$.

\begin{theorem}  \label{thm:BH}
Let $\TP\in\R^{[3,n]}$ be a nonnegative tensor, let 
$$
   \triangle(\TP) = 
   \max_{i_1,j_1,k_1,i_2,j_2,k_2}\, \frac{\TP_{i_1j_1k_1}\TP_{i_2j_2k_2}}{\TP_{i_1j_2k_1}\TP_{i_2j_1k_2}},
$$
and let 
$\kappa(\TP)  = \tanh(\frac14 \log\triangle(\TP))$. 
Then 
$$
   d_H(\TP xy, \TP x'y') \leq
   \kappa(\TP)d_H(x,x') + \kappa(\TP^S)d_H(y,y')\, .
$$
\end{theorem}
From Theorem \ref{thm:BH} we immediately derive a formula for the higher-order Birkhoff ergodicity coefficient for stochastic tensors, and the corresponding analogous of Theorem \ref{thm:ergodic_thm_tensor}. Precisely, we have
the following result.

\begin{corollary}   \label{cor:BH}
Let $\TP\in\R^{[3,n]}$ be a stochastic tensor and let 
$$
   \Tau_H(\TP) = 2\,\kappa(\TP+\TP^S) =
   2\, \tanh(\textstyle\frac14 \log\widehat\triangle(\TP))
$$
where
$$
   \widehat\triangle(\TP) = 
   \max_{i_1,j_1,k_1,i_2,j_2,k_2}\, \frac{(\TP_{i_1j_1k_1}+\TP_{i_1k_1j_1})(\TP_{i_2j_2k_2}+\TP_{i_2k_2j_2})}{(\TP_{i_1j_1k_2}+\TP_{i_1k_2j_1})(\TP_{i_2j_2k_1}+\TP_{i_2k_1j_2})}\, .
$$
If $\Tau_H(\TP)<1$ then there exists a unique  $Z$-eigenvector $x\in \mathcal S_1$ such that $\TP xx = x$ and the higher-order power iteration $x_{t+1}=\TP x_tx_t$ converges to $x$ for any starting point $x_0\in S_1$.
\end{corollary}

\begin{proof}
Consider the $S$-symmetric
tensor $\b Q = \frac12 (\b P + \b P^S)$.
Note that $\triangle(\b Q) = \widehat\triangle(\b P)$ 
and thus 
$\kappa(\b Q) = \frac 1 2 \Tau_H(\b P)$.
Therefore, using the identity $\b Pxx = \b Qxx$, which holds for all $x\in\R^n$, the triangle inequality for $d_H$ and Theorem \ref{thm:BH}, we have 
\begin{align*}
    d_H(\TP xx, \TP yy) & = d_H(\b Qxx,\b Qyy)  \leq d_H(\b Q xx, \b Q xy) + d_H(\b Q xy, \b Q yy)   \\
    & \leq \kappa(\b Q) [ d_H(x,x)+ d_H(x,y) + 
           d_H(x,y) + d_H(y,y) ]   = \Tau_H(\TP) d_H(x,y) .
\end{align*}
This shows that $x\mapsto \TP xx$ is a contraction with respect to the Hilbert metric. As $(\mathcal S_1,d_H)$ is a complete metric space, the proof continues as that of Theorem \ref{thm:ergodic_thm_tensor}.
\end{proof}

Note that, similarly to the $1$-norm case, $\Tau_H(\TP)=0$ if and only if $\TP$ has rank one, that is, $\Tau_H$ is proper. 
However, while $\Tau_H(\TP)=2$ for any tensor $\TP$ not of rank one and having at least one zero entry, 
$\Tau(\TP)$ can be smaller than one even for sparse tensors. For example, if $\TP$ is the tensor
$$
\TP = \frac 1 2 \left[ \begin{bmatrix}
0 & 1 & 1 \\
1 & 0 & 0 \\
1 & 1 & 1
\end{bmatrix}\begin{bmatrix}
0 & 1 & 1 \\
1 & 0 & 1 \\
1 & 1 & 0
\end{bmatrix}\begin{bmatrix}
0 & 0 & 0 \\
2 & 0 & 1 \\
0 & 2 & 1
\end{bmatrix}\right]
$$
then one easily verifies that $\Tau_H(\b P) =2$, while $\Tau(\b P) = 1/2$. 

The left panel of Figure \ref{fig:compare_taus} scatter plots these two coefficients computed on a set of ten thousand random stochastic $n\times n \times n$ tensors with size $n$  between $2$ and $10$. 
In the matrix case it is well known that, for any stochastic matrix $P$ it holds $\tau_1(P)\leq \tau_H(P)$, see \cite[\S 3.4]{SenetaBook}.
While the  numerical comparison 
shown in Figure \ref{fig:compare_taus} suggests the inequality $\Tau(\TP)\leq \Tau_H(\TP)$, an explicit comparison between the $1$-norm and the Birkhoff higher-order coefficients $\Tau(\TP)$ and $\Tau_H(\TP)$, for general tensors,  is out of scope and is left open to future work.

\subsection{Li and Ng's coefficients}

Given a stochastic tensor $\b P\in\R^{[3,n]}$, consider the following quantities
introduced in \cite{Li+NLAA13,LiNg}:
\begin{gather}
    \gamma(\TP) = \min_{\II\subset [n]} 
    \Big\{ \min_{k} \Big( \min_{j\in \II} \sum_{i\not\in \II}\b P_{ijk}+
    \min_{j\not\in \II}\sum_{i\in \II}\b P_{ijk} \Big)  +\min_{j}\Big( \min_{k\in \II} \sum_{i\not\in \II}\b P_{ijk}+
    \min_{k\not\in \II}\sum_{i\in \II}\b P_{ijk} \Big)  \Big\}  \label{eq:gamma}
     \\
    \delta(\TP) = \min_{\II\subset [n]} 
    \Big(  \min_{j,k} \sum_{i\not\in \II}\b P_{ijk}+
    \min_{j,k}\sum_{i\in \II}\b P_{ijk} \Big) .
    \label{eq:delta3}
\end{gather}
Li and Ng proved in \cite{LiNg}
two conditions for the uniqueness of the stationary distribution and the convergence of the iteration $x_{t+1} = \TP x_tx_t$ 
in terms of the entries of $\TP$, 
that we review in the following. 

\begin{theorem}[\cite{LiNg}]   \label{thm:LiNg}
Let $\TP\in\R^{[3,n]}$ be a stochastic tensor.
If $\gamma(\TP) > 1$ then there exists an unique solution $x\in\SS$
of the equation $x = \TP xx$.
Moreover, 
the iteration $x_{t+1} = \TP x_tx_t$ converges to~$x$.
\end{theorem}

As $\gamma(\TP) \geq 2\delta(\TP)$, the following consequence is immediate.

\begin{corollary}   \label{thm:LiNg2}
Let $\TP\in\R^{[3,n]}$ be a stochastic tensor.
If $\delta(\TP) > 1/2$ then 
all the claims in the preceding theorem are true.
\end{corollary}

Moreover, we recall from \cite[Thm.\ 4]{Li+NLAA13}
the three-mode case of a perturbation bound for the stationary probability vector of a 
stochastic tensor of order $m > 2$.

\begin{theorem}[\cite{Li+NLAA13}]  \label{thm:Li+NLAA13}
Let $\TP$ and its perturbation $\TP'$ be two stochastic tensors in $\R^{[3,m]}$.
If $\delta(\TP) > 1/2$ then the stochastic solution of $x = \TP xx$ is unique, and for 
any stochastic vector $x'$ such that $x' = \TP' x'x'$ it holds
$$
   \|x-x'\|_1 \leq \frac{\|\TP - \TP'\|_1}{2\delta(\TP) - 1} .
$$
\end{theorem}

In the sequel, we aim  to compare the above
results with the ones we proved in 
the previous sections. First, we prove a special characterization of $\delta(\TP)$
in \eqref{eq:delta3}, which provides an explicit formula for $\delta(\TP)$ in terms of the entries of $\TP$.

\begin{lemma}   \label{lem:chardelta}
Let $\TP\in\R^{[3,n]}$ be a stochastic tensor. Then
$$
   \delta(\TP) =  1 - \frac12 
   \max_{j_1,j_2,k_1,k_2} \| \TP e_{j_1}e_{k_1} - \TP e_{j_2}e_{k_2} \|_1 .
$$
\end{lemma}

\begin{proof}
First, note that for any zero-sum vector $y\in\R^n$ it holds $\sum_i |y_i| = 2\max\big\{ \sum_{i\in\II} y_i : \II \subseteq \{1,\dots,n\} \big\}$.
Let $j_1,j_2,k_1,k_2$ be fixed. 
Then we have
\begin{align*}
   \| \TP e_{j_1}e_{k_1} -& \TP e_{j_2}e_{k_2} \|_1  = 
   \sum_{i} | \TP_{ij_1k_1} - \TP_{ij_2k_2} |   =
   2 \max_{\mathcal{I}\subset[n]} \sum_{i\in\mathcal{I}} 
   ( \TP_{ij_1k_1} - \TP_{ij_2k_2} ) \\
   & = 2 \max_{\mathcal{I}\subset[n]}\Big( 1- \sum_{i\notin\mathcal{I}} \TP_{ij_1k_1} - \sum_{i\in\mathcal{I}} \TP_{ij_2k_2} 
   \Big)  = 2 - 2\min_{\mathcal{I}\subset[n]}\Big( \sum_{i\notin\mathcal{I}} \TP_{ij_1k_1} + \sum_{i\in\mathcal{I}} \TP_{ij_2k_2} 
   \Big) .
\end{align*}
Therefore
\begin{align*}
   1 - \frac12 \max_{j_1,j_2,k_1,k_2} 
   \| \TP e_{j_1}e_{k_1} - \TP e_{j_2}e_{k_2} \|_1 
   & =  \frac12 \min_{j_1,j_2,k_1,k_2} 
   \big( 2 - \| \TP e_{j_1}e_{k_1} - \TP e_{j_2}e_{k_2} \|_1 \big) \\
   & = \min_{j_1,j_2,k_1,k_2} \min_{\mathcal{I}\subset[n]}
   \Big( \sum_{i\notin\mathcal{I}} \TP_{ij_1k_1} + \sum_{i\in\mathcal{I}} \TP_{ij_2k_2} 
   \Big) , 
\end{align*}
which coincides with \eqref{eq:delta3}, after rearranging terms.
\end{proof}

\begin{figure}[t]
    \centering
    \includegraphics[width=\textwidth]{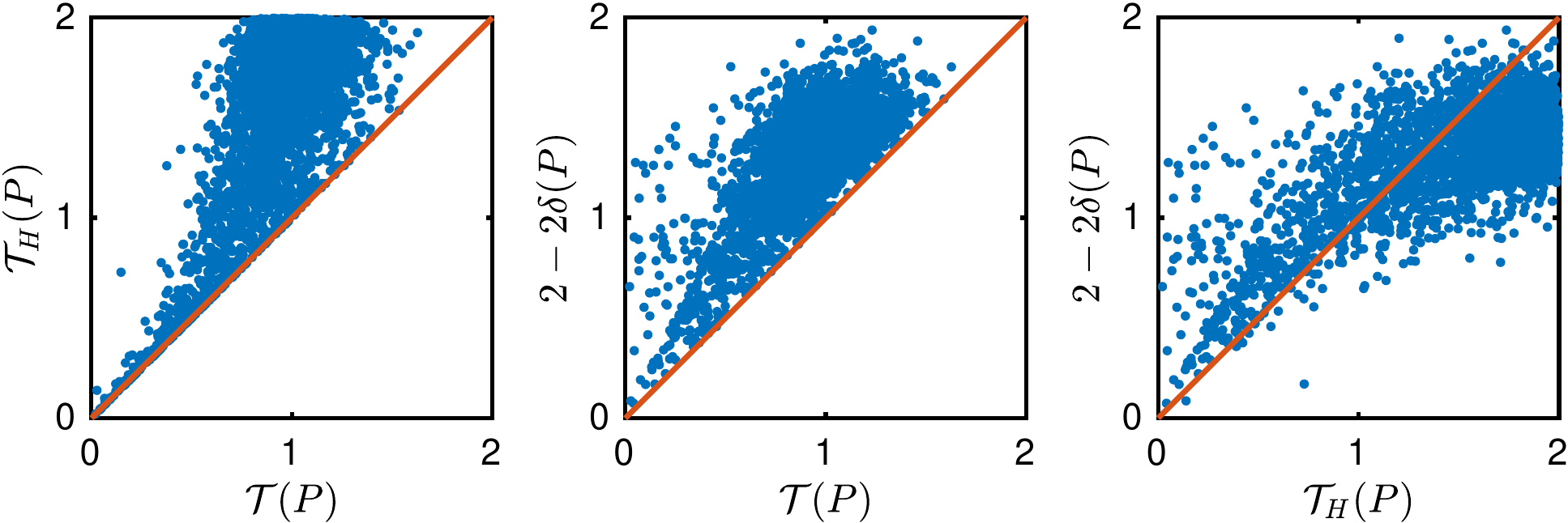}
    \caption{Scatter plot of different coefficients over 10,000 random $n\times n \times n$ stochastic tensors $\TP$ with size $n$ chosen uniformly at random within $\{2,\dots,10\}$. }
    \label{fig:compare_taus}
\end{figure}

Using the characterization of $\delta(\TP)$ in the preceding lemma, the following theorem compares $\delta(\TP)$ and $\gamma(\TP)$ with the higher-order ergodic coefficient $\Tau(\TP)$:
\begin{theorem}
Let $\TP\in\R^{[3,n]}$ be stochastic. 
Then $\Tau(\TP) \leq 2 - 2\delta(\TP)$. Moreover, if
$\TP = \TP^S$ then $2 - \gamma(\TP) \leq \Tau(\TP)$.
\end{theorem}

\begin{proof}
The formulas \eqref{eq:TauL} and \eqref{eq:TauR} can be rewritten 
as
$$
    \Tau_L(\TP) = \frac12 \max_{j,k_1,k_2} \| \TP e_j(e_{k_1}-e_{k_2}) \|_1 , \qquad
    \Tau_R(\TP) = \frac12 \max_{j_1,j_2,k} \| \TP (e_{j_1}-e_{j_2})e_k \|_1 ,
$$
respectively. Using the preceding formulas and \eqref{eq:TauLR}
it is immediate to obtain 
\begin{align*}
   \max_{j_1,j_2,k_1,k_2} \| \TP e_{j_1}e_{k_1} - \TP e_{j_2}e_{k_2} \|_1 
    \geq 2 \max\{ \Tau_L(\TP) , \Tau_R(\TP)\}  \geq \Tau_L(\TP) + \Tau_R(\TP)
   \geq \Tau(\TP) .
\end{align*}
From Lemma \ref{lem:chardelta} we conclude $1 - \delta(\TP) \geq  \frac12 \Tau(\TP)$
and this proves the first part of the claim.
Furthermore, using the symmetry $\TP = \TP^S$, the formulas 
\eqref{eq:gamma} and \eqref{eq:Taus2} simplify to
\begin{align*}
    \gamma(\TP) & =  2 \min_{\II\subset [n]} 
    \min_{j}\Big( \min_{k\in \II} \sum_{i\not\in \II}\b P_{ijk}+
    \min_{k\not\in \II}\sum_{i\in \II}\b P_{ijk} \Big) \\
   \Tau(\TP) & = 2 - 2 \min_{\II\subset[n]} \min_j \Big(
   \min_{k} \sum_{i\in\II} \TP_{ijk} 
   + \min_{k} \sum_{i\notin\II} \TP_{ijk} \Big) .
\end{align*}
The inequality $2 - \gamma(\TP) \leq \Tau(\TP)$ follows, and the proof is complete.
\end{proof}

We conclude with several important remarks that we obtain as a consequence of the preceding results.

First, notice that the requirement $\delta(\TP) > \frac12$ appearing in Theorem \ref{thm:Li+NLAA13}
is stronger than the one of Theorem \ref{thm:cond3}, 
namely, if $\delta(\TP) > \frac12$ holds then 
$\Tau(\TP) < 1$ must hold as well. Moreover, 
$2\delta(\TP) - 1 \geq 1 - \Tau(\TP)$. Thus the right hand side of Theorem \ref{thm:Li+NLAA13} is larger than the one of Theorem \ref{thm:cond3}. This shows that 
Theorem \ref{thm:cond3}  is an improvement over Theorem \ref{thm:Li+NLAA13}.

On the other hand, the condition $\gamma(\TP) > 1$ is weaker than
$\Tau(\TP) < 1$. 
Hence, the hypothesis in Theorem \ref{thm:LiNg}
ensuring uniqueness of 
the solution of $x = \TP xx$
and convergence of the higher-order
power method can be more general than the one in Theorem \ref{thm:ergodic_thm_tensor}, at least when $\TP = \TP^S$. 
Additionally,
it is important to point out that the inequality $\Tau(\TP) < 1$ can be checked using $\mathcal{O}(n^4)$ arithmetic operations,
while the computation of $\gamma(\TP)$ requires the solution of a nontrivial combinatorial optimization problem, which is in general significantly more expensive.

The central and the rightmost
panels of Figure \ref{fig:compare_taus} compare numerically, via scatter plots, the condition  $2-2\delta(\TP)<1$ and the ergodicity conditions $\Tau(\b P)<1$ and $\Tau_H(\b P)<1$ obtained via the higher-order ergoditicity coefficients, on a test set of $10,000$ randomly generated tensors with varying size.

\section{Examples}   \label{sec:Z-eig-examples}

We conclude with a number of example applications of 
Theorem \ref{thm:ergodic_thm_tensor}.
The examples here below further demonstrate   
the usefulness of the newly introduced higher-order ergodicity coefficients
in a variety of contexts.

\subsection{Multilinear PageRank}

Given a stochastic tensor $\b P$, a scalar $0<\alpha<1$ and a probability vector $v\in \SS$, the  multilinear PageRank is a solution of the equation
\begin{equation}   \label{eq:mlpr}
   \alpha\b Pxx + (1-\alpha)v = x\, .
\end{equation}
This definition has been introduced by Gleich, Lim, and Yu \cite{MLPR} 
in analogy to the renowned  Google's PageRank vector, defined as the solution of $\alpha Px+(1-\alpha)v=x$ where
$P$ is a stochastic transition probability matrix. 
Pursuing that analogy, the solution of \eqref{eq:mlpr} gives the stationary probability of a stochastic process that with probability $\alpha$ behaves like the second-order Markov chain \eqref{eq:PM2} 
and with probability $1-\alpha$ teleports to a random state chosen according to the discrete density $v$.

A detailed analysis of the possibly multiple nonnegative solutions to \eqref{eq:mlpr} is provided by Meini and Poloni in \cite{MeiPol18}. 
They also discuss various first- and second-order iterative methods to compute a solution to  \eqref{eq:mlpr}. In particular,  fixed-point type methods are often a choice of preference, due to their inexpensive iterations and simple implementation. Also, these types of methods can be easily extrapolated achieving fast converge rates, see \cite{cipolla2019extrapolation}.  However, in practice one is interested in values of $\alpha$ not too far from $1$ but, unlike the matrix case, requiring $\alpha<1$ is not enough to ensure the uniqueness of the multilinear PageRank nor the convergence of the fixed-point iterates. In the original paper \cite{MLPR}, the condition $\alpha < 1/2$ is proved to be sufficient to ensure both these properties \eqref{eq:mlpr}. 
More recently, a tighter sufficient condition for the uniqueness of the multilinear PageRank has been proved by Li et al. \cite{li2017uniqueness}, in terms of the following quantity,
$$
  \theta(\TP,\sigma) = 
   \max_{j,k_1,k_2} \sum_i |\TP_{ijk_1} - \sigma_i| 
   + |\TP_{ik_2j} - \sigma_i| ,
$$
where $\sigma$ is any real vector. Precisely, Theorems 1 and 2 in \cite{li2017uniqueness} show that if there exists $\sigma\in \mathbb R^n$ such that $\alpha\, \theta(\TP,\sigma)<1$, then \eqref{eq:mlpr} has a unique nonnegative solution and the fixed-point iteration for \eqref{eq:mlpr} converges to such a solution. 

Theorem \ref{thm:ergodic_thm_tensor} provides a new condition that 
improves the range of values of $\alpha$ for which we can guarantee both the  uniqueness of a nonnegative  solution of \eqref{eq:mlpr} and the convergence of the associated fixed point iteration, as shown by the following result.

\begin{corollary}   \label{cor:mlPR}
If $\alpha\Tau(\TP) < 1$ then \eqref{eq:mlpr} has a unique solution $x\in\SS$. 
Moreover, the fixed point iteration $x_{t+1} = \alpha\TP x_tx_t + (1-\alpha) v$
converges linearly to $x$, with a convergence rate of at least $\alpha\Tau(\TP)$.
Finally, it holds $\| x - v \|_1 \leq 2\alpha$.
\end{corollary}

\begin{proof}
Let $\b V\in\R^{[3,n]}$ be the rank-one tensor 
$\b V_{ijk} = v_i$. 
Since $\b V xx = v$ for any $x\in\SS$,
the equation \eqref{eq:mlpr} can rewritten as $x = \TP_\alpha xx$ where
\begin{equation}\label{eq:Palpha}
   \TP_\alpha = \alpha \TP + (1-\alpha) \b V .
\end{equation}
By Theorem \ref{thm:ergodic_thm_tensor}, the condition $\Tau(\TP_\alpha) < 1$ guarantees uniqueness of the solution and convergence of the fixed point iteration. However,  
$
   \Tau(\TP_\alpha) \leq \alpha\Tau(\TP) + (1-\alpha)\Tau(\b V )
   = \alpha\Tau(\TP) ,
$  
due to the fact that $\Tau(\b V ) = 0$, as noted in Remark \ref{rem:rk1}.

Finally, note that the vector $v$ is characterized by the identity $v = \b V vv = \TP_0 vv$.
Hence, by considering $\TP_\alpha$ as a perturbation of $\b V = \TP_0$,
from Theorem \ref{thm:cond3} we get 
$$
   \| x - v \|_1 \leq \frac{1}{1 - \Tau(\b V)} \| \TP_\alpha - \b V \|_1
   = \alpha \| \TP - \b V \|_1 
   \leq 2\alpha
$$
since $\Tau(\b V) = 0$, and the proof is complete.
\end{proof}

Note that the condition for the uniqueness given by Corollary \ref{cor:mlPR} is always an improvement with respect to the one of \cite{li2017uniqueness}. In fact, using the formula \eqref{eq:Tau} for $\Tau(\TP)$, for any $\sigma \in \mathbb R^n$ we have 
\begin{align*}
   2\Tau(\TP) = & 
   \max_{j,k_1,k_2} \sum_i |\TP_{ijk_1} - \TP_{ijk_2} 
   + \TP_{ik_1j} - \TP_{ik_2j} + 2 \sigma_i - 2 \sigma_i | \\
   \leq & \max_{j,k_1,k_2} \sum_i |\TP_{ijk_1} - \sigma_i | + 
   |\TP_{ijk_2} - \sigma_i | + |\TP_{ik_1j} - \sigma_i| + 
   | \TP_{ik_2j} - \sigma_i | \\
   \leq & \Big[ \max_{j,k_1,k_2} \sum_i 
   |\TP_{ijk_1} - \sigma_i | + |\TP_{ik_2j} - \sigma_i | 
   \Big] + \Big[ \max_{j,k_1,k_2} \sum_i 
   |\TP_{ijk_2} - \sigma_i | + |\TP_{ik_1j} - \sigma_i | 
   \Big] \\
   {} = {} & 2 \theta(\TP,\sigma) .
\end{align*}

\begin{figure}[t]\label{fig:mlpr_1}
\centering
\includegraphics[width=.57\textwidth]{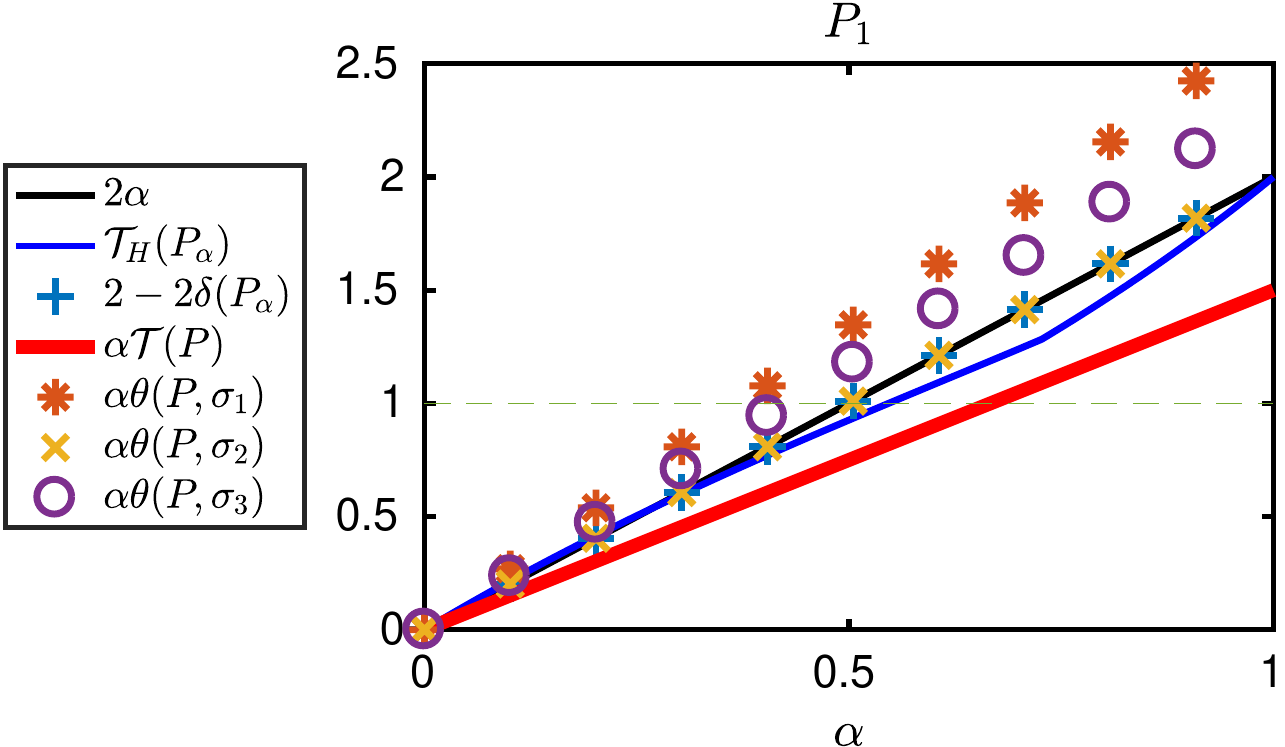}\includegraphics[width=.42\textwidth]{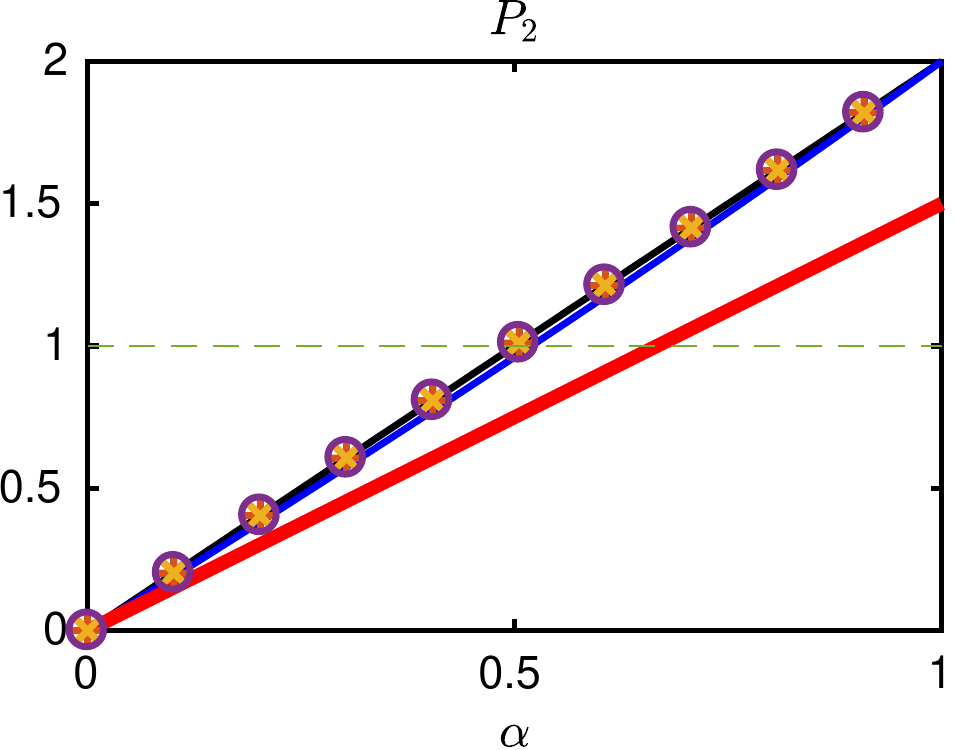}
\caption{This figure compares the results in \cite[Thm.\ 5.1]{MLPR}, Corollary \ref{cor:BH}, Theorem \ref{thm:ergodic_thm_tensor}, \cite[Cor.\ 1 \& 2]{li2017uniqueness} by comparing the values of $2\alpha$, $\Tau_H(\TP_\alpha)$, $2-2\delta(\TP_\alpha)$, $\alpha \Tau(\TP)$, $\alpha\theta(\TP,\sigma_k)$, $k=1,2,3$,   where $\TP_\alpha$ is defined as in \eqref{eq:Palpha},  the vectors $\sigma_1,\sigma_2,\sigma_3$ defined as in  \eqref{eq:sigmas} and  $\TP$ is either of the example tensors $\TP_1$ and $\TP_2$ of \eqref{eq:example_tensors}.  }
\end{figure}

In order to illustrate how the various conditions differ in practice, we consider two small example tensors  borrowed from
\cite{MLPR}
\begin{align}
\begin{aligned}\label{eq:example_tensors}
    \TP_1 &= \left[\begin{bmatrix}1/3 & 1/3 & 1/3\\1/3 & 1/3 & 1/3\\1/3 & 1/3 & 1/3 \end{bmatrix} \begin{bmatrix}1/3 & 0 & 0\\1/3 & 0 & 1/2\\1/3 & 0 & 1/2\end{bmatrix}\begin{bmatrix} 0&0&0\\ 1&0&1 \\ 0&1&0 \end{bmatrix}  \right]\, ,\\
    \TP_2 &= \left[\begin{bmatrix}
    0 & 0 & 1/3\\0 & 0 & 1/3\\1 & 1 & 1/3 \end{bmatrix} \begin{bmatrix}1/3 & 0 & 0\\1/3 & 0 & 0\\1/3 & 1 & 1\end{bmatrix}\begin{bmatrix} 1/2 & 1/2 & 1/2\\ 0 & 0 & 1/2 \\ 1/2 & 1/2 & 0 \end{bmatrix}  \right] .
    \end{aligned}
\end{align}

Figure \ref{fig:mlpr_1} compares the range of values of $\alpha$ that  guarantee uniqueness of the multilinear PageRank and convergence of the corresponding fixed-point iteration  for the two tensors  $\TP_1$ and $\TP_2$, according to the original Theorem 5.1 in \cite{MLPR}, Corollary \ref{cor:BH},  Theorem \ref{thm:ergodic_thm_tensor} and Theorem 1 in \cite{li2017uniqueness}. For the latter result, we show the value of the quantities  $\alpha\,  \theta(\TP,\sigma_k)$, $k=1,2,3$ obtained with the three choices of vectors
\begin{equation}\label{eq:sigmas}
    (\sigma_1)_i = \max_{jk}\TP_{ijk}, \quad (\sigma_2)_i = \min_{jk}\TP_{ijk}, \quad \sigma_3 = \frac{\sigma_1+\sigma_2}2\, ,
\end{equation}
as proposed in Corollaries 1 and 2 in the same paper.
The interesting ranges are those where the corresponding graphs stay below the dashed line.

\subsection{Triangle-based PageRank on networks}

Random walks  are an important tool for exploratory network analysis. For example, they are at the basis of widely used methods for local clustering, link prediction and network centrality. 
The typical random walk on a network is a Markov process where the probability to move from a node $i$ to a node $j$ is proportional to the number of outgoing edges leaving from $i$. 
This classical first-order process only takes into account pairwise node-node relationships. However, recent work has highlighted that many important network features arise by exploiting the interaction of larger groups of nodes acting together, see e.g., \cite{AHTproceedings,Benson2015TensorSC}. 
In order to account for this type of second-order node interaction, we can consider a second-order stochastic process on the network, where the
probability to move to a node $i$ depends on the number of triangles that point towards $i$. 
We show in this section how the higher-order ergodicity coefficients for stochastic tensors help dealing with triangle-based random walks.

Let  $G=(V,E)$ be an undirected graph, with $V = \{1,\dots,n\}$, and consider the tensor
$$ 
   \b T_{ijk}  = \begin{cases} \frac{1}{\triangle(j,k)} & \text{if $i,j,k$ form a triangle in $G$} \\
   0 & \text{otherwise,} \end{cases}
$$ 
where $\triangle(j,k)$ is the number of triangles that contain both nodes $j$ and $k$. This tensor is the triangle-based version of the transition matrix of the standard random walk in $G$,
$$
   A_{ij} = \begin{cases}
   \frac 1 {d(j)} & \text{if $i,j$ form an edge in $G$} \\
   0 & \text{otherwise,} \end{cases}
$$
where $d(i) = \sum_j A_{ij}$ is the degree of node $i$. Clearly, $\b T$ has many vanishing columns as in general two nodes $j,k\in V$ may not participate in any triangle in $G$. 
In that case, we set $\b T_{ijk} = 1/n$ for all $i = 1,\ldots,n$. Similarly, we set $A_{ij}=1/n$ for all $i=1,\dots,n$ if $j$ is an isolated node in $G$ (i.e.\ if the $j$-th column of $A$ is zero). 

Now, define the tensor $\b A$ as $\b A_{ijk} = A_{ij}$ and, for $\beta\in[0,1]$, consider the stochastic tensor 
\begin{equation}   \label{eq:trianglebasedtensor}
   \TP = \beta \b T + (1-\beta) \b A .
\end{equation}
This construction has been considered for example in \cite{AHTproceedings,MLPR}, within the multilinear PageRank equation \eqref{eq:mlpr}, in order to combine the standard and the triangle-based random walks on real-world networks. 
The next result specializes Theorem \ref{thm:ergodic_thm_tensor} to the multilinear PageRank problem associated with the tensor in \eqref{eq:trianglebasedtensor} and, additionally, provides a bound on the distance between the solution to that problem and the standard PageRank vector. 

\begin{corollary}   \label{cor:mlPRtriangle}
Let $\TP$ be defined as in \eqref{eq:trianglebasedtensor} and let $\gamma = \alpha(1+\beta)$.
If $\gamma < 1$ then \eqref{eq:mlpr} has a unique solution $x\in\SS$, 
and the fixed point iteration $x_{t+1} = \alpha\TP x_tx_t + (1-\alpha) v$
converges linearly to $x$, with a convergence rate of at least $\gamma$.

Moreover, let $z\in\SS$ be the solution of the ordinary PageRank problem corresponding to the transition matrix $A$ and the teleportation vector $v$,
\begin{equation}  \label{eq:PRclassic}
   z = \alpha A z + (1-\alpha) v .
\end{equation}
Then,
$$
   \| x - z \|_1 \leq 
   \frac{\alpha\beta}{1-\alpha} \| \b T - \b A\|_1 .
$$
\end{corollary}

\begin{proof}
Observe that, as $\Tau(\b A) = \tau_1(A) \leq 1$,
we have the trivial upper bound
$\Tau(\TP) \leq \beta + 1$,
hence the first part of the claim follows from Corollary \ref{cor:mlPR}.
Moreover, simple passages allow us to recast $x$ as the stochastic solution of the equation 
$x = \TP_{\alpha,\beta} xx$ where
$$
   \TP_{\alpha,\beta} = \alpha\beta \b T + \alpha(1-\beta)\b A + (1-\alpha) \b V 
$$
and $\b V$ is the rank-one tensor $\b V_{ijk} = v_i$.
Analogously, the vector $z$ can also be considered as the stochastic solution of
$
   z = (\alpha \b A + (1-\alpha) \b V) zz
$, 
that is $z = \TP_{\alpha,0} zz$.
We have $\Tau(\TP_{\alpha,0}) \leq \alpha$, 
hence from Theorem \ref{thm:cond3} we get
$$
   \| x - z \|_1 \leq \frac{1}{1 - \alpha}
   \| \TP_{\alpha,\beta} - \TP_{\alpha,0} \|_1 
   = \frac{\alpha\beta}{1 - \alpha}
   \| \b T - \b A \|_1 ,
$$
which completes the proof.
\end{proof}

Together with Corollary \ref{cor:mlPR},
the previous result shows that small values of $\alpha$ and $\beta$  produce a multilinear PageRank vector that does not differ sensibly from the ordinary PageRank vector.  
Note that the uniqueness and convergence condition $\gamma < 1$ in Corollary \ref{cor:mlPRtriangle}
can be fulfilled by any value $\alpha < 1/(1+\beta)$. This condition is evidently less restrictive than the better known inequality $\alpha < 1/2$. This is one of several implications of Corollary \ref{cor:mlPRtriangle}. Below we consider an example  real-wold network to further showcase the advantages of that corollary.

\begin{figure}[t]
    \centering
    \includegraphics[width=\textwidth,clip,trim=1.5cm .1cm  1cm .5cm]{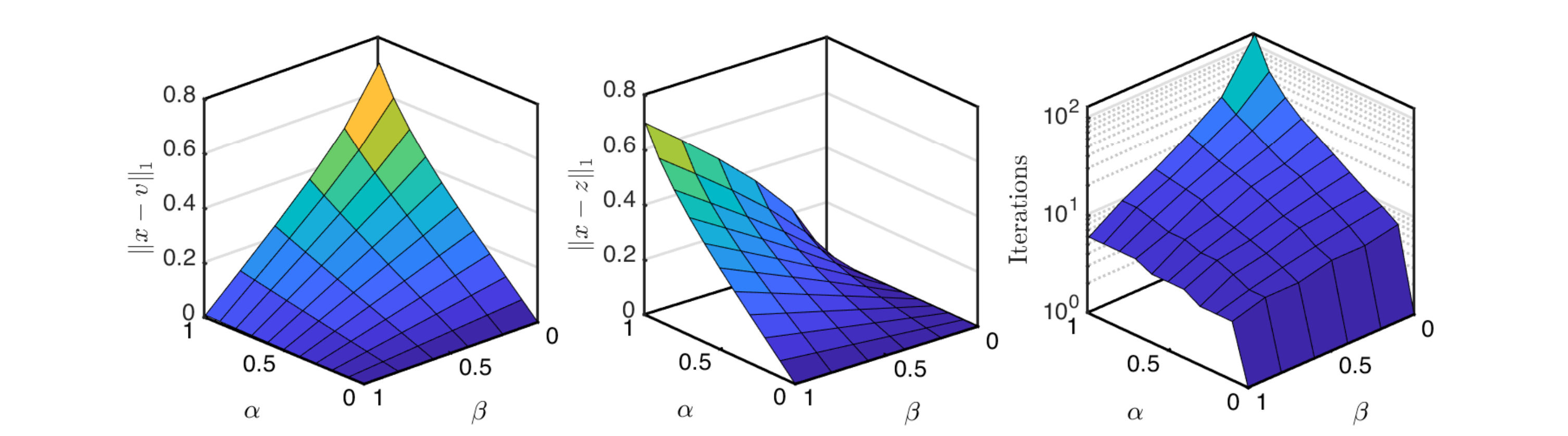}
    \caption{Triangle-based PageRank analysis on the {\tt socfb-Carnegie49} Facebook network, with varying $\alpha$ and $\beta$. 
    Left to right: $\| x - v\|_1$, $\| x - z\|_1$, and number of iterations $x_{t+1}=\b P_{\alpha,\beta}x_t x_t$ to reach  $\|x_{t+1}-x_t\|<10^{-8}$.}
    \label{fig:Carnegie49}
\end{figure}

The {\tt socfb-Carnegie49} network is a Facebook graph considered for example in the study of the social structure of Facebook users \cite{traud2012social},
and available online on NetworkRepository \cite{NetworkRepository}. The graph has $6637$ nodes and $249967$ undirected edges. 
The largest connected component consists of $6621$ vertices, and the triangle tensor $\b T$ has $13860318$ nonzero entries.
Figure \ref{fig:Carnegie49}
shows the results of a number of multilinear PageRank problems with coefficients $\alpha$ and $\beta$ varying in $[0,1]$. The equation $x = \alpha \b Pxx + (1-\alpha)v$, with $\TP = \beta \b T + (1-\beta) \b A$ as in \eqref{eq:trianglebasedtensor} and with uniform teleportation vector $v = \uno/n$
has been solved via the fixed point iteration $x_{t+1} = \alpha \b Px_tx_t + (1-\alpha)v$
endowed with the stopping criterion $\|x_{t+1}-x_t\|_1 < 10^{-8}$.

The leftmost panel in Figure \ref{fig:Carnegie49} shows the distance $\|x-v\|_1$, whereas the central panel shows $\|x-z\|_1$ where $z$ is the usual PageRank vector, defined as in \eqref{eq:PRclassic}, with the same $\alpha$ value chosen for the multilinear version. The iteration number to convergence is shown in the rightmost panel. 

While the overall behavior of $\|x-z\|_1$ reflects the estimate in Corollary \ref{cor:mlPRtriangle},
the panel on the left shows that $x$ approaches $v$ not only when $\alpha \approx 0$ but also when $\beta$ is large. 
This is due to the fact that
the triangle tensor $\b T$ has many zero columns. 
Consequently, the vast majority of the columns of $\TP$ coincide with the uniform vector $v$ and, for an arbitrary vector $x\in\SS$, the product $\TP xx$ is in general very close to $v$, in the sense that the $1$-norm of the vector $r(x) = \TP xx - v$ is rather small.
With this notation we  obtain
\begin{align*}
   x - v & = \alpha(1-\beta) \{\b A xx - v\} + \alpha\beta r(x) 
    = \alpha(1-\beta)\{Ax - v\} + \alpha\beta r(x) .
\end{align*}
Broadly speaking, when $\b T$ is very sparse $\|r(x)\|_1$ is usually negligible, and we can adopt the estimate $\|x-v\|_1 \approx \alpha(1-\beta) \|A x - v\|_1$. This approximation justifies the small error $\|x-v\|_1$ observed when $\beta \approx 1$.

\begin{figure}[t]
    \centering
    \includegraphics[width=\textwidth,clip,trim=0cm 0cm  0cm 0cm]{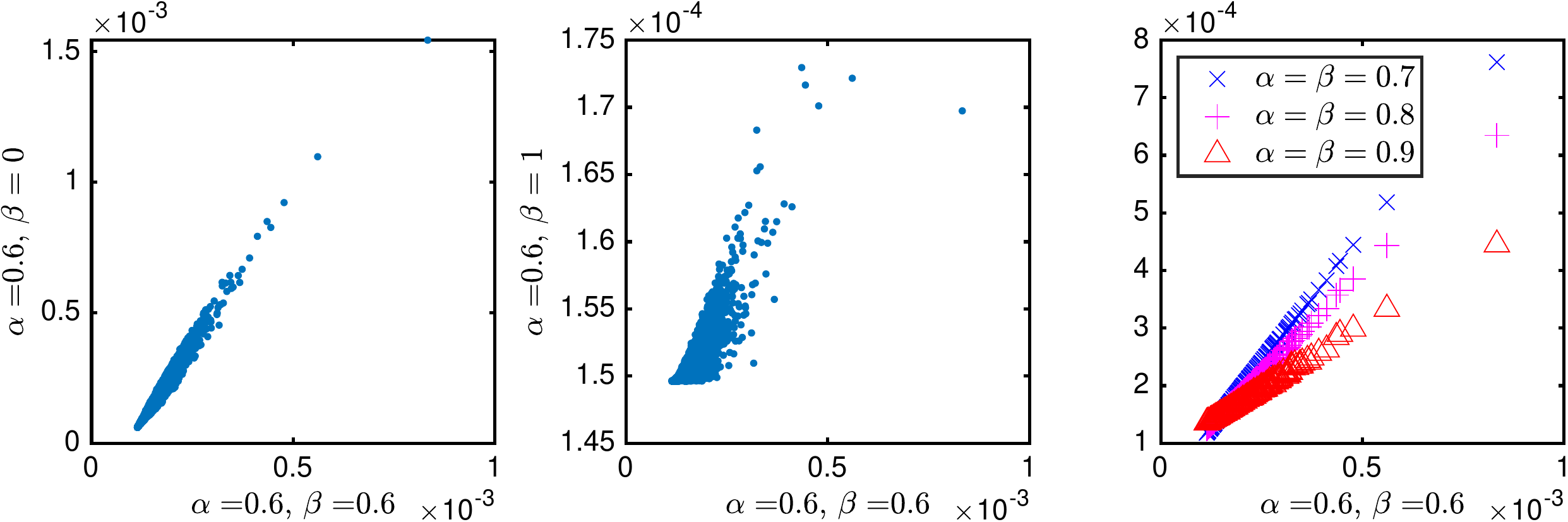}
    \caption{Scatter plots of triangle-based PageRank vectors of the {\tt socfb-Carnegie49} network for different choices of the parameters $\alpha$ and $\beta$. 
    From left to right: comparisons between the solution with $\alpha = \beta = 0.6$ and the standard PageRank vector (left);
    the purely triangle-based solution $\alpha=0.6$ and $\beta = 1$ (center); and other solutions with $\alpha = \beta$ (right).}
    \label{fig:Carnegie49_2}
\end{figure}

In conclusion, the most informative results are obtained when both $\alpha$ and $\beta$ are neither too small nor too close to $1$. Extensive numerical experiments we performed on several real-world networks suggest the ``reference'' choice $\alpha , \beta \approx 0.6$. 
These values yield a good balance between first- and second-order information, fulfill the condition $\alpha(1+\beta) < 1$ in Corollary \ref{cor:mlPRtriangle} and ensure a fast convergence of the fixed point iteration.
As an illustration, in Figure \ref{fig:Carnegie49_2} we compare via scatter plots the reference solution for $\alpha=\beta=0.6$ with other multilinear PageRank vectors for different values of $\alpha$ and $\beta$ chosen as follows: in the leftmost panel we compare the reference solution against the standard PageRank vector with $\alpha=0.6$; in the central panel the solution for $\alpha=\beta=0.6$ is compared against the purely triangle-based case $\alpha=0.6$ and $\beta=1$; in the rightmost panel the vector corresponding to $\alpha=\beta=0.6$ is scatter plotted against the solution for the three choices $\alpha=\beta \in \{0.7,0.8,0.9\}$.  The first two panels  show the sensitivity of the solution with respect to different choices of $\beta$ highlighting, in particular, the importance of both edge- and triangle-based walks in the graph. The last panel on the right, instead, shows that the reference solution for $\alpha=\beta=0.6$ highly correlates with other numerical solutions obtained with larger values of the coefficients $\alpha=\beta$. This illustrates that larger choices of the coefficients $\alpha=\beta$ essentially do not alter the information on the nodes, but require a much larger iteration count.

\subsection{Higher-order shifted power method and lazy random walk}

Let $\TP$ be symmetric, that is,
$\TP = \TP^{\langle \pi\rangle}$ for every permutation $\pi$ of $\{1,2,3\}$.
In \cite{KoMa11}, Kolda and Mayo analyzed the convergence of the
``shifted symmetric higher-order power method''
\begin{equation}\label{eq:SHOPM}
   \hat x_{t+1} = \TP x_tx_t + \alpha x_t , \qquad
   x_{t+1} = \frac{\hat x_{t+1}}{\|\hat x_{t+1}\|_2} .
\end{equation}
Their starting point is the optimization of the cubic form $f(x) = x^T (\TP xx)$
over the sphere $x^Tx = 1$, whose stationary points are, for symmetric tensors, $Z$-eigenvectors of $\b P$ 
related to the best symmetric rank-one approximation of $\TP$. The coefficient $\alpha$
can be chosen positive or negative, in order to 
make the modified function $f(x) + \alpha x^Tx$ convex or concave, respectively.
Using fixed point theory, the authors of \cite{KoMa11} prove that,
given an appropriate shift $\alpha$ the iterates in \eqref{eq:SHOPM} 
generically converge to some $Z$-eigenvector.
The shifting technique has been considered also for tensors that are not symmetric. For example, it has been considered in the framework of the multilinear PageRank \cite{MLPR} or in the case of $\ell^p$-eigenvalue computation \cite{gautier2018unifying}.

Let the coefficient $\beta(\b P)$ be defined as $\beta(\TP) = 2 \max_{\|x\|_2=1} \rho(\TP x)$, where $\rho(\TP x)$ denotes the spectral radius of the matrix $\TP x$.
One of the main results from \cite{KoMa11} is that, if 
$\TP$ is symmetric and $\alpha > \beta(\TP)$, then the method \eqref{eq:SHOPM} converges to some stationary point of $f$, which is a $Z$-eigenvector of $\b P$.

If $\b P$ is stochastic (but not necessarily symmetric), 
then it is natural to replace the sphere $x^Tx = 1$ with the simplex $\SS$ and the vector $2$-norm with the $1$-norm. With these replacements, 
and other minor notation changes,
the iteration \eqref{eq:SHOPM} boils down~to 
\begin{equation}   \label{eq:SHOPM1}
   x_{t+1} = \sigma \TP x_tx_t + (1-\sigma) x_t  
   \qquad  \sigma\in(0,1) ,
\end{equation}
which, for an initial stochastic vector $x_0$, will remain 
in $\SS$ throughout.
This iteration coincides with the higher-order power method $x_{t+1}=\TP_\sigma x_tx_t$ for the 
``shifted tensor''
$$
   \TP_\sigma = \sigma \TP + (1-\sigma) \b E \, ,
$$
where $\b E$ is 
any tensor such that $\b Exx = x$,  for all $x\in\SS$. For example,   $\b E$  can be chosen as  a  convex combination of the left and right identities $\b E^L$ and $\b E^R$, defined in \eqref{eq:Etensors}.

Note that $\TP_\sigma$ is stochastic, for any choice of $\sigma\in(0,1)$ and thus the iteration \eqref{eq:SHOPM1}
can be interpreted as a form of higher-order lazy random walk. In fact, recall that if $P\in \R^{n\times n}$ is a stochastic matrix, then the Markov chain associated with $\sigma P + (1-\sigma) I$ is called lazy random walk, as it describes a walker that, with probability $\sigma$ 
performs a transition according to $P$, and
remains in its current state otherwise. 
Hence,  we can use Theorem \ref{thm:ergodic_thm_tensor} to provide a condition on $\sigma$, in terms of the entries of $\b P$, that guarantees global convergence 
of the shifted power method \eqref{eq:SHOPM1}.

Even though the higher-order ergodicity coefficient $\Tau(\b P)$  of the original tensor may be larger than one, suitable values of $\sigma$ can ensure that Theorem \ref{thm:ergodic_thm_tensor} holds for $\TP_\sigma$. 
In fact, it is interesting to note that the function $\sigma\mapsto\Tau(\TP_\sigma)$ is continuous, piecewise linear and convex, with $\Tau(\TP_0) = \Tau(\b E) = 1$.
As $x=\TP_\sigma xx$ if and only if $x= \TP xx$, we deduce that

\begin{corollary}\label{cor:shifted_pm}
If $\TP$ is stochastic and $\Tau(\TP_\sigma)<1$ for some $\sigma \geq 0$, then $\TP$ has a unique positive $Z$-eigenvector $x\in \SS$ and the method \eqref{eq:SHOPM} converges to $x$, for any starting point $x_0\in\SS$, with a convergence rate of at least $\Tau(\TP_\sigma)$.
\end{corollary}

\begin{proof}
The claim follows straightforwardly from Theorem \ref{thm:ergodic_thm_tensor} applied to $\TP_\sigma$.
\end{proof}

\begin{figure}[t]
\centering
\includegraphics[width=.8\textwidth]{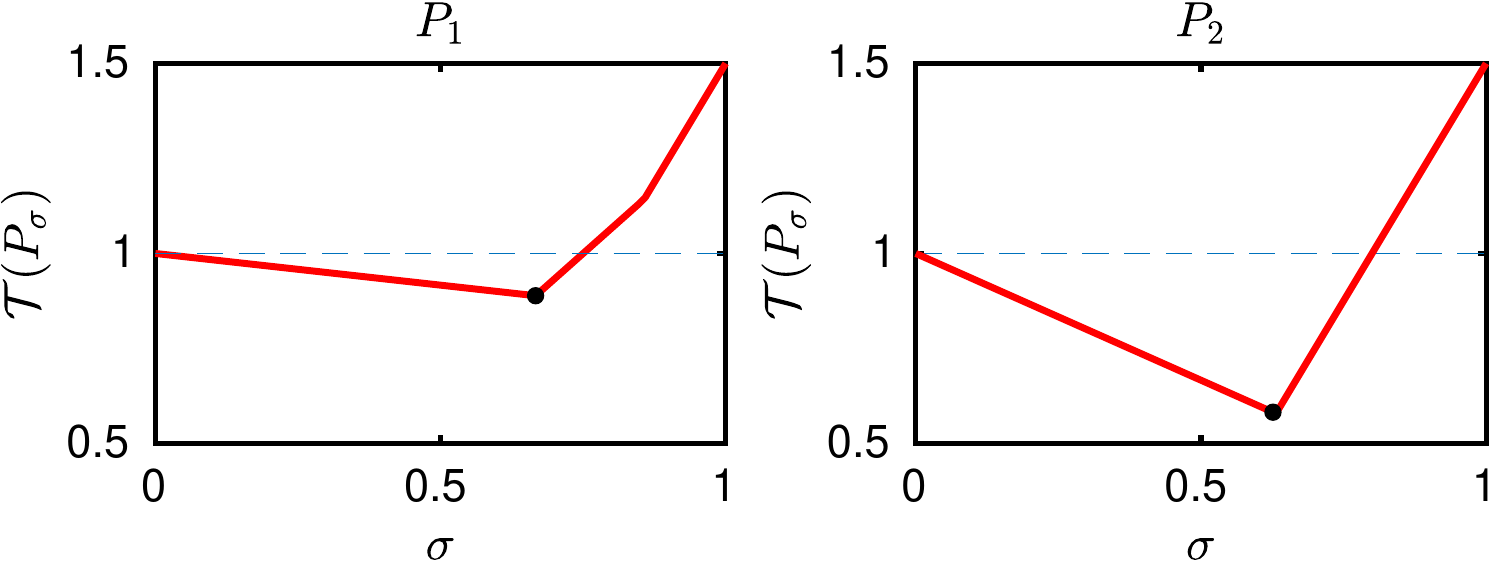}
\caption{Variation of $\Tau(\TP_\sigma)$ as $\sigma$ varies within $[0,1]$, for the two example tensors in \eqref{eq:example_tensors}. }\label{fig:shifted}
\end{figure}

In Figure \ref{fig:shifted} we show the value of $\Tau(\TP_\sigma)$ as a function of $\sigma$, for the two example tensors \eqref{eq:example_tensors} and for the choice $\b E = \frac 1 2 (\b E^L + \b E^R)$. Notice that for both the examples shown there exists an optimal $\sigma_*$ such that $\min_\sigma \Tau(\TP_\sigma)=\Tau(\TP_{\sigma_*})<1$. Thus, although the  higher-order ergodicity coefficient $\Tau(\b P)$ of the original tensor is larger than one, by Corollary \ref{cor:shifted_pm} there exists a unique positive $x\in \SS$ such that $x=\TP xx$ and we can compute it with a method that converges as $\|x_{t+1}-x\|_1\leq \Tau(\TP_{\sigma^*})^t\|x_0-x\|_1$, for an arbitrary $x_0\in \SS$.

\section{Conclusions}

This work adds to the long and continuing history of applications of tensor methods to data science by providing a novel analysis of the long-term behaviour of higher-order stochastic processes  governed by stochastic tensors. 
These types of processes are used in a large number of network science and data mining applications due to their ability to improve the underlining models and offer additional valuable insights. Even though stationary distributions of these processes are often required, fundamental mathematical questions such as the uniqueness of the distribution and the convergence behavior of the stochastic process remain unanswered. In fact, this is a relatively newly born and actively growing research area, with many open questions.

Following a natural extension of the widely used ergodicity coefficients for Markov chains,  we have introduced a new family of higher-order ergodicity coefficients for higher-order processes that provides new and easily computable conditions to ensure existence, uniqueness and convergence towards the corresponding stationary distribution.  The proposed analysis adds to previous work on uniqueness of $Z$-eigenvectors of stochastic tensors \cite{bozorgmanesh2016convergence,Li+NLAA13,LiNg} and non-negative tensors in general \cite{chang2008perron,Fried,gautier2017perron,gautier2018unifying} by providing new  conditions that are either less restrictive or are computationally easier to verify, or both.

\section*{Acknowledgments}
The main results of this work have been developed during a visiting period that F.T. has spent at the Department of Mathematics, Computer Science and Physics
of the University of Udine, Italy. He would like to thank the department and D.F. for the warm hospitality he received during that period.


\begin{thebibliography}{10}

\bibitem{anandkumar2014tensor}
{\sc A.~Anandkumar, R.~Ge, D.~Hsu, S.~M. Kakade, and M.~Telgarsky}, {\em Tensor
  decompositions for learning latent variable models}, The Journal of Machine
  Learning Research, 15 (2014), pp.~2773--2832.

\bibitem{arrigo2019non}
{\sc F.~Arrigo, D.~J. Higham, and V.~Noferini}, {\em Non-backtracking
  {PageRank}}, Journal of Scientific Computing,  (2019), pp.~1--19.

\bibitem{AHTproceedings}
{\sc F.~Arrigo, D.~J. Higham, and F.~Tudisco}, {\em A framework for
  second-order eigenvector centralities and clustering coefficients}, Proc. R.
  Soc. A, 476 (2020).

\bibitem{arrigo2019HITS}
{\sc F.~Arrigo and F.~Tudisco}, {\em Multi-dimensional, multilayer, nonlinear
  and dynamic {HITS}}, in Proceedings of the 2019 SIAM International Conference
  on Data Mining, SIAM, 2019, pp.~369--377.

\bibitem{Ben97}
{\sc M.~Bena\"im}, {\em Vertex-reinforced random walks and a conjecture of
  {P}emantle}, Ann. Probab., 25 (1997), pp.~361--392.

\bibitem{BGL17}
{\sc A.~Benson, D.~F. Gleich, and L.-H. Lim}, {\em The spacey random walk: {A}
  stochastic process for higher-order data}, SIAM Rev., 59 (2017),
  pp.~321--345.

\bibitem{benson2019three}
{\sc A.~R. Benson}, {\em Three hypergraph eigenvector centralities}, SIAM
  Journal on Mathematics of Data Science, 1 (2019), pp.~293--312.

\bibitem{Benson2015TensorSC}
{\sc A.~R. Benson, D.~F. Gleich, and J.~Leskovec}, {\em Tensor spectral
  clustering for partitioning higher-order network structures}, Proceedings of
  the 2015 SIAM International Conference on Data Mining,  (2015), pp.~118--126.

\bibitem{BerRaf02}
{\sc A.~Berchtold and A.~Raftery}, {\em The mixture transition distribution
  model for high-order {M}arkov chains and non-{G}aussian time series},
  Statist. Sci., 17 (2002), pp.~328--356.

\bibitem{birkhoff1}
{\sc G.~Birkhoff}, {\em Extensions of {J}entzsch's theorem}, Transactions of
  the American Mathematical Society, 85 (1957), pp.~219--227.

\bibitem{bouguet2018fluctuations}
{\sc F.~Bouguet and B.~Cloez}, {\em Fluctuations of the empirical measure of
  freezing {M}arkov chains}, Electronic Journal of Probability, 23 (2018),
  pp.~Paper No. 2, 31.

\bibitem{bozorgmanesh2016convergence}
{\sc H.~Bozorgmanesh and M.~Hajarian}, {\em Convergence of a transition
  probability tensor of a higher-order {M}arkov chain to the stationary
  probability vector}, Numerical Linear Algebra with Applications, 23 (2016),
  pp.~972--988.

\bibitem{chang2013uniqueness}
{\sc K.~Chang and T.~Zhang}, {\em On the uniqueness and non-uniqueness of the
  positive {Z}-eigenvector for transition probability tensors}, Journal of
  Mathematical Analysis and Applications, 408 (2013), pp.~525--540.

\bibitem{chang2008perron}
{\sc K.-C. Chang, K.~Pearson, and T.~Zhang}, {\em Perron--{F}robenius theorem
  for nonnegative tensors}, Communications in Mathematical Sciences, 6 (2008),
  pp.~507--520.

\bibitem{chierichetti2012web}
{\sc F.~Chierichetti, R.~Kumar, P.~Raghavan, and T.~Sarlos}, {\em Are web users
  really markovian?}, in Proceedings of the 21st international conference on
  World Wide Web, ACM, 2012, pp.~609--618.

\bibitem{cipolla2019extrapolation}
{\sc S.~Cipolla, M.~Redivo-Zaglia, and F.~Tudisco}, {\em Extrapolation methods
  for fixed-point multilinear {P}age{R}ank computations}, Numerical Linear
  Algebra with Applications, 27 (2020), p.~e2280.

\bibitem{cui2019uniqueness}
{\sc L.-B. Cui and Y.~Song}, {\em On the uniqueness of the positive
  {Z}-eigenvector for nonnegative tensors}, Journal of Computational and
  Applied Mathematics, 352 (2019), pp.~72--78.

\bibitem{de2000best}
{\sc L.~De~Lathauwer, B.~De~Moor, and J.~Vandewalle}, {\em On the best rank-1
  and rank-(r 1, r 2,..., rn) approximation of higher-order tensors}, SIAM
  journal on Matrix Analysis and Applications, 21 (2000), pp.~1324--1342.

\bibitem{dobrushin1956central}
{\sc R.~L. Dobrushin}, {\em Central limit theorem for nonstationary {M}arkov
  chains. {I}, {II}}, Theory of Probability \& Its Applications, 1 (1956),
  pp.~65--80, 329--383.

\bibitem{BHNB}
{\sc S.~P. Eveson and R.~D. Nussbaum}, {\em An elementary proof of the
  {B}irkhoff--{H}opf theorem}, in Mathematical Proceedings of the Cambridge
  Philosophical Society, vol.~117, Cambridge University Press, 1995,
  pp.~31--55.

\bibitem{Fried}
{\sc S.~Friedland, S.~Gaubert, and L.~Han}, {\em {P}erron--{F}robenius theorem
  for nonnegative multilinear forms and extensions}, Linear Algebra Appl., 438
  (2013), pp.~738--749.

\bibitem{gautier2018contractivity}
{\sc A.~Gautier and F.~Tudisco}, {\em The contractivity of cone-preserving
  multilinear mappings}, Nonlinearity, 32 (2019), pp.~4713--4728.

\bibitem{gautier2017perron}
{\sc A.~Gautier, F.~Tudisco, and M.~Hein}, {\em The {P}erron--{F}robenius
  theorem for multihomogeneous mappings}, SIAM J. Matrix Analysis Appl., 40
  (2019), pp.~1179--1205.

\bibitem{gautier2018unifying}
{\sc A.~Gautier, F.~Tudisco, and M.~Hein}, {\em A unifying
  {P}erron--{F}robenius theorem for nonnegative tensors via multihomogeneous
  maps}, SIAM J. Matrix Analysis Appl., 40 (2019), pp.~1206--1231.

\bibitem{MLPR}
{\sc D.~F. Gleich, L.-H. Lim, and Y.~Yu}, {\em Multilinear {P}age{R}ank}, SIAM
  J. Matrix Anal. Appl., 36 (2015), pp.~1507--1541.

\bibitem{hillar2013most}
{\sc C.~J. Hillar and L.-H. Lim}, {\em Most tensor problems are {NP-hard}},
  Journal of the ACM (JACM), 60 (2013), p.~45.

\bibitem{HuQi14}
{\sc S.~Hu and L.~Qi}, {\em Convergence of a second order {M}arkov chain},
  Appl. Math. Comput., 241 (2014), pp.~183--192.

\bibitem{hu2016computing}
{\sc S.~Hu, L.~Qi, and G.~Zhang}, {\em Computing the geometric measure of
  entanglement of multipartite pure states by means of non-negative tensors},
  Physical Review A, 93 (2016), p.~012304.

\bibitem{Ergodic}
{\sc I.~C.~F. Ipsen and T.~M. Selee}, {\em Ergodicity coefficients defined by
  vector norms}, SIAM J. Matrix Anal. Appl., 32 (2011), pp.~153--200.

\bibitem{Knopp}
{\sc K.~Knopp}, {\em Infinite sequences and series}, Dover Publications, Inc.,
  New York, 1956.

\bibitem{KoMa11}
{\sc T.~G. Kolda and J.~R. Mayo}, {\em Shifted power method for computing
  tensor eigenpairs}, SIAM J. Matrix Anal. Appl., 32 (2011), pp.~1095--1124.

\bibitem{kolokoltsov2010nonlinear}
{\sc V.~N. Kolokoltsov}, {\em Nonlinear {M}arkov processes and kinetic
  equations}, vol.~182, Cambridge University Press, 2010.

\bibitem{krzakala2013spectral}
{\sc F.~Krzakala, C.~Moore, E.~Mossel, J.~Neeman, A.~Sly, L.~Zdeborov{\'a}, and
  P.~Zhang}, {\em Spectral redemption in clustering sparse networks},
  Proceedings of the National Academy of Sciences, 110 (2013),
  pp.~20935--20940.

\bibitem{LiZhang15}
{\sc C.-K. Li and S.~Zhang}, {\em Stationary probability vectors of
  higher-order {M}arkov chains}, Linear Algebra Appl., 473 (2015),
  pp.~114--125.

\bibitem{Li+NLAA13}
{\sc W.~Li, L.-B. Cui, and M.~K. Ng}, {\em The perturbation bound for the
  {P}erron vector of a transition probability tensor}, Numer. Linear Algebra
  Appl., 20 (2013), pp.~985--1000.

\bibitem{li2017uniqueness}
{\sc W.~Li, D.~Liu, M.~K. Ng, and S.-W. Vong}, {\em The uniqueness of
  multilinear {PageRank} vectors}, Numerical Linear Algebra with Applications,
  24 (2017), p.~e2107.

\bibitem{LiNg}
{\sc W.~Li and M.~K. Ng}, {\em On the limiting probability distribution of a
  transition probability tensor}, Linear Multilinear Algebra, 62 (2014),
  pp.~362--385.

\bibitem{mei2010divrank}
{\sc Q.~Mei, J.~Guo, and D.~Radev}, {\em Divrank: the interplay of prestige and
  diversity in information networks}, in Proceedings of the 16th ACM SIGKDD
  international conference on Knowledge discovery and data mining, ACM, 2010,
  pp.~1009--1018.

\bibitem{MeiPol18}
{\sc B.~Meini and F.~Poloni}, {\em Perron-based algorithms for the multilinear
  {P}age{R}ank}, Numer. Linear Algebra Appl., 25 (2018), pp.~e2177, 15.

\bibitem{nassar2019pairwise}
{\sc H.~Nassar, A.~R. Benson, and D.~F. Gleich}, {\em Pairwise link
  prediction}, arXiv preprint arXiv:1907.04503,  (2019).

\bibitem{Ng2011MultiRank}
{\sc M.~K. Ng, X.~Li, and Y.~Ye}, {\em Multirank: co-ranking for objects and
  relations in multi-relational data}, in Proceedings of the 17th ACM SIGKDD
  Conference on Knowledge Discovery and Data Mining, 2011, pp.~1217--1225.

\bibitem{Pemantle92}
{\sc R.~Pemantle}, {\em Vertex-reinforced random walk}, Probab. Theory Related
  Fields, 92 (1992), pp.~117--136.

\bibitem{QiLuo}
{\sc L.~Qi and Z.~Luo}, {\em Tensor {A}nalysis: {S}pectral {T}heory and
  {S}pecial {T}ensors}, SIAM, 2017.

\bibitem{qi2008d}
{\sc L.~Qi, Y.~Wang, and E.~X. Wu}, {\em {D}-eigenvalues of diffusion kurtosis
  tensors}, Journal of Computational and Applied Mathematics, 221 (2008),
  pp.~150--157.

\bibitem{Raftery85}
{\sc A.~E. Raftery}, {\em A model for high-order {M}arkov chains}, J. Roy.
  Statist. Soc. Ser. B, 47 (1985), pp.~528--539.

\bibitem{RafTav94}
{\sc A.~E. Raftery and S.~Tavar\'e}, {\em Estimation and modelling repeated
  patterns in high order {M}arkov chains with the mixture transition
  distribution model}, Journal of the Royal Statistical Society. Series C., 43
  (1994), pp.~179--199.

\bibitem{ragnarsson2012block}
{\sc S.~Ragnarsson and C.~F. Van~Loan}, {\em Block tensor unfoldings}, SIAM
  Journal on Matrix Analysis and Applications, 33 (2012), pp.~149--169.

\bibitem{NetworkRepository}
{\sc R.~A. Rossi and N.~K. Ahmed}, {\em The network data repository with
  interactive graph analytics and visualization}, in Proceedings of the
  Twenty-Ninth AAAI Conference on Artificial Intelligence, AAAI’15, AAAI
  Press, 2015, p.~4292–4293, \url{http://networkrepository.com} (accessed
  2020-04-26).

\bibitem{rosvall2014memory}
{\sc M.~Rosvall, A.~V. Esquivel, A.~Lancichinetti, J.~D. West, and
  R.~Lambiotte}, {\em Memory in network flows and its effects on spreading
  dynamics and community detection}, Nature communications, 5 (2014), p.~4630.

\bibitem{Sab19}
{\sc M.~Saburov}, {\em Ergodicity of {$\bf p$}-majorizing nonlinear {M}arkov
  operators on the finite dimensional space}, Linear Algebra Appl., 578 (2019),
  pp.~53--74.

\bibitem{SenetaBook}
{\sc E.~Seneta}, {\em Non-negative matrices and {M}arkov chains},
  Springer-Verlag, 1981.

\bibitem{Seneta88}
{\sc E.~Seneta}, {\em Perturbation of the stationary distribution measured by
  ergodicity coefficient}, Advances in Applied Probability, 20 (1988),
  pp.~228--230.

\bibitem{traud2012social}
{\sc A.~L. Traud, P.~J. Mucha, and M.~A. Porter}, {\em Social structure of
  {F}acebook networks}, Phys. A, 391 (2012), pp.~4165--4180.

\bibitem{tudisco2015note}
{\sc F.~Tudisco}, {\em A note on certain ergodicity coefficients}, Special
  Matrices, 3 (2015), pp.~175--185.

\bibitem{williams2019effects}
{\sc O.~E. Williams, F.~Lillo, and V.~Latora}, {\em Effects of memory on
  spreading processes in non-markovian temporal networks}, New Journal of
  Physics, 21 (2019), p.~043028.

\bibitem{WuChu17}
{\sc S.-J. Wu and M.~T. Chu}, {\em Markov chains with memory, tensor
  formulation, and the dynamics of power iteration}, Appl. Math. Comput., 303
  (2017), pp.~226--239.

\end{thebibliography}

\end{document}